\documentclass[a4paper,12pt]{scrartcl}
\usepackage{amsmath,amssymb,latexsym,amsthm,enumerate}
\usepackage{xcolor}
\usepackage{graphicx}
\usepackage{epstopdf}
\usepackage[T1]{fontenc}
\usepackage[utf8]{inputenc}
\usepackage[english]{babel}
\newcommand*{\wh}{\widehat}
\newcommand*{\wt}{\widetilde}
\newcommand*{\ol}{\overline}

\newcommand*{\eps}{\varepsilon}

\newcommand*{\N}{\mathbb{N}}
\newcommand*{\R}{\mathbb{R}}

\newcommand*{\IZ}{\mathbb{Z}}
\newcommand*{\Znn}{\bbN_0}

\newcommand*{\bbN}{\mathbb N}

\newcommand*{\bbR}{\mathbb R}

\newcommand*{\cA}{\mathcal{A}}
\newcommand*{\cB}{\mathcal{B}}

\newcommand*{\cF}{\mathcal{F}}

\newcommand*{\loc}{\mathrm{loc}}
\newcommand*{\ucp}{\mathrm{ucp}}

\newcommand*{\Law}{\operatorname{Law}}

\newcommand{\be}{\begin{eqnarray*}}
\newcommand{\ee}{\end{eqnarray*}}
\newcommand{\ben}{\begin{eqnarray}}
\newcommand{\een}{\end{eqnarray}}
\newcommand{\bi}{\begin{itemize}}
\newcommand{\ei}{\end{itemize}}

\newtheorem{theo}{Theorem}[section]

\newtheorem{lemma}[theo]{Lemma}
\newtheorem{propo}[theo]{Proposition}
\newtheorem{corollary}[theo]{Corollary}

\theoremstyle{definition}
\newtheorem{ex}[theo]{Example}

\newtheorem{remark}[theo]{Remark}

\newcounter{numpar}[section]

\title{A functional limit theorem for\\ coin tossing Markov chains}

\author{Stefan Ankirchner\thanks{%
Stefan Ankirchner, Institute of Mathematics, University of Jena, Ernst-Abbe-Platz 2, 07745 Jena, Germany. \emph{Email:} s.ankirchner@uni-jena.de, \emph{Phone:} +49 (0)3641 946275.
}
\and Thomas Kruse\thanks{Institute of Mathematics, University of Gie{\ss}en, Arndtstr.~2, 35392 Gießen, Germany.
\emph{Email:} thomas.kruse@math.uni-giessen.de, \emph{Phone:} +49 (0)641 9932102.}
\and Mikhail Urusov\thanks{%
Mikhail Urusov, Faculty of Mathematics, University of Duisburg-Essen, Thea-Leymann-Str.~9, 45127 Essen, Germany.
\emph{Email:} mikhail.urusov@uni-due.de, \emph{Phone:} +49 (0)201 1837428.
}}

\begin{document}

\maketitle

\begin{abstract}
We prove a functional limit theorem for Markov chains that, in each step, move up or down by a possibly state dependent constant with probability $1/2$, respectively. The theorem entails that the law of every one-dimensional
regular continuous strong Markov process in natural scale can be approximated with such Markov chains arbitrarily well.
The functional limit theorem applies, in particular, to Markov processes that cannot be characterized as solutions to stochastic differential equations. 
Our results allow to practically approximate such processes with irregular behavior;
we illustrate this with Markov processes exhibiting sticky features, e.g., sticky Brownian motion and a Brownian motion slowed down on the Cantor set.   


\end{abstract}

\section*{Introduction}

Let $(\xi_k)_{k \in \bbN}$ be an iid sequence of random variables, on a probability space with a measure $P$, satisfying $P(\xi_1 = \pm 1) = \frac12$. 
Given $y\in \R$, $h \in (0,\infty)$ and a function $a_h:\R \to \R$, 
we denote by $(X^h_{kh})_{k \in \bbN_0}$ the Markov chain defined by
\begin{equation}\label{eq:def_X_intro}
X^h_0 = y
\quad\text{and}\quad
X^h_{(k+1)h} = X^h_{kh}+ a_h(X^h_{kh}) \xi_{k+1}, \quad \text{ for } k \in \bbN_0. 
\end{equation}
We choose as the Markov chain's index set the set of non-negative multiples of $h$ because we interpret $h$ as the length of a time step. We extend $(X^h_{kh})_{k \in \bbN_0}$ to a continuous-time process by linear interpolation, i.e., we set
\begin{equation}\label{eq:lin_interpo_intro}
X^h_t = X^h_{\lfloor t/h \rfloor h} + (t/h - \lfloor t/h \rfloor) (X^h_{(\lfloor t/h \rfloor +1)h} - X^h_{\lfloor t/h \rfloor h}), \qquad t\in[0,\infty). 
\end{equation}
Let $\ol h \in (0,\infty)$ and let $(a_h)_{h \in (0,\ol h)}$ be a family of real functions and $(X^h)_{h \in (0,\ol h)}$ the associated family of extended Markov chains defined as in \eqref{eq:lin_interpo_intro}. A fundamental problem of probability theory  is to find conditions on $(X^h)_{h \in (0,\ol h)}$ such that  the laws of the processes $X^h$, $h \in (0,\ol h)$, converge in some sense as $h\to 0$.  In this article we provide an asymptotic condition on the family $(a_h)_{h \in (0,\ol h)}$ guaranteeing that the laws of the processes $X^h$, $h \in (0,\ol h)$, converge as $h\to 0$ to the law of a one-dimensional
regular continuous strong Markov process (in the sense of Section~VII.3 in \cite{RY} or Section~V.7 in \cite{RogersWilliams}). In what follows we use the term
{\itshape general diffusions} for the latter class of processes. 
Recall that a general diffusion $Y = (Y_t)_{t \in [0,\infty)}$ has a state space that is an open, half-open or closed interval $I \subseteq \R$. We denote by $I^\circ=(l,r)$ the interior of $I$, where $-\infty\leq l<r\leq \infty$. 
Moreover, the law of any general diffusion is uniquely characterized by its speed measure $m$ on $I$, its scale function and its boundary behavior. 
Throughout the introduction we assume that $Y$ is in natural scale and that 
every accessible boundary point is absorbing (see the beginning of Section \ref{sec:MC approx} and Section~\ref{sec:reflecting} on how to incorporate 
diffusions in general scale
and with reflecting boundary points).
This setting covers, in particular, solutions of driftless SDEs with discontinuous  and fast growing diffusion coefficient
(see Section~\ref{sec:sdes}) and also diffusions with sticky features (see Section~\ref{sec:examples}),
which cannot be modeled by SDEs whenever a sticky point is located in the interior of the state space.

Our main result, Theorem~\ref{main thm sec 1}, shows that
if a family of functions $(a_h)_{h\in (0,\ol h)}$ satisfies for all $y\in I^\circ$, $h\in (0, \ol h)$ the equation 
\begin{equation}\label{eq:charac_sf_intro}
\frac{1}{2}\int_{(y-a_h(y),y+a_h(y))}(a_h(y)-|u-y|)\, m(du)=h,
\end{equation}
with a precision of order $o(h)$ uniformly in $y$
over compact subsets of $I^\circ$
(see Condition~(A) below for a precise statement), then
the associated family $(X^h)_{h\in (0,\ol h)}$ converges in distribution, as $h\to 0$, to the general diffusion $Y$ with speed measure $m$.
We show that for every general diffusion a family of functions $(a_h)_{h\in (0,\ol h)}$
satisfying \eqref{eq:charac_sf_intro} exists implying that every general diffusion can be approximated by a Markov chain of the form \eqref{eq:def_X_intro}. Equation~\eqref{eq:charac_sf_intro} dictates how to compute the functions $(a_h)_{h\in (0,\ol h)}$ and therefore paves the way to approximate the distribution of a general diffusion numerically (see, e.g., Section \ref{sec:cantor_bm}). 

The central idea in the derivation of Equation~\eqref{eq:charac_sf_intro} is to embed for every $h \in (0,\ol h)$ the Markov chain $(X^h_{kh})_{k \in \bbN_0}$ into $Y$ with a sequence of stopping times. To explain this idea assume for the moment that the state space is $I=\R$. For every 
$h \in (0,\ol h)$ let $\tau^h_0=0$ and then
recursively define $\tau^h_{k+1}$ as the first time $Y$ exits the interval $(Y_{\tau^h_k}-a_h(Y_{\tau^h_k}), Y_{\tau^h_k}+a_h(Y_{\tau^h_k}))$ after  $\tau^h_{k}$.
It follows that the discrete-time process $(Y_{\tau^h_k})_{k\in \N_0}$ has the same law as the Markov chain $(X^h_{kh})_{k\in \N_0}$. 
Instead of controlling now directly the spatial errors $|Y_{\tau^h_k}-Y_{kh}|$,
we first analyze the temporal errors $|\tau^h_k-kh|$, $k\in \N_0$. 
We show that for every $y\in \R$, $a\in [0,\infty)$ the expected time it takes $Y$ started in $y$ to leave the interval 
$(y-a,y+a)$ is equal to $\frac{1}{2}\int_{(y-a,y+a)}(a-|u-y|)\, m(du)$. In particular,
if $a_h$ satisfies \eqref{eq:charac_sf_intro} for all $y\in I$, it 
follows that for all $k\in \N_0$ the time lag $\tau^h_{k+1}-\tau^h_k$ between two consecutive stopping times is in expectation equal to $h$.
In this case we refer to $(X^h_{kh})_{k\in \N_0}$ as Embeddable Markov Chain with
Expected time Lag $h$ (we write shortly $(X^h_{kh})_{k\in \N_0} \in \text{EMCEL}(h)$). 

For some diffusions $Y$ one can construct EMCEL approximations explicitly (see, e.g., Section~\ref{sec:examples}). 
For cases where \eqref{eq:charac_sf_intro} cannot be solved in closed form, we perform a perturbation analysis
and show that it suffices to find for all $h\in \ol h$, $y\in I^\circ$ a number $a_h(y)$ satisfying \eqref{eq:charac_sf_intro} with an error of order $o(h)$ uniformly in $y$ belonging to compact subsets of $I^\circ$.
We prove that for the associated stopping times $(\tau^h_k)_{k\in \N_0}$ the temporal errors
$|\tau^h_k-kh|$, $k\in \N_0$, converge to $0$ as $h\to 0$ in every $L^\alpha$-space, $\alpha\in [1,\infty)$. This 
%implies convergence 
%of $(Y_{\tau^h_k})_{k\in \N_0}$ 
%to $(Y_{kh})_{k\in \N_0}$ in probability  and 
 ultimately implies convergence 
 of $(X^h)_{h\in (0,\ol h)}$ to $Y$ in distribution as $h\to 0$.

To illustrate the benefit of the perturbation analysis,
we construct in Section~\ref{sec:cantor_bm}
approximations for a Brownian motion
slowed down on the Cantor set (see Figure~\ref{fig:cbm}).
Moreover, we note that our main result, Theorem~\ref{main thm sec 1},
is not only applicable to perturbations of the EMCEL approximation
but can also be used to derive new convergence results
for other approximation methods such as, e.g., weak Euler schemes
(see Corollary~\ref{cor:euler}).

The idea to use embeddings in order to prove a functional limit theorem goes back to Skorokhod. In the seminal book \cite{skoro} scaled random walks are embedded into Brownian motion in order to prove Donsker's invariance principle. In \cite{aku-aihp} we embed Markov chains into the solution process of an SDE and prove a functional limit theorem where the limiting law is that of the SDE. 
In \cite{skoro} and \cite{aku-aihp} the approximating Markov chains have to be embeddable with a sequence of stopping times $(\tau_k)_{k \in \mathbb{N}_0}$ such that the expected distance between two consecutive stopping times is {\itshape exactly} equal to $h$, the time discretization parameter. In contrast, in the present article we require that the expected distance between consecutive embedding stopping times is only approximately equal to $h$. We show that for the convergence of the laws it is sufficient to require that the difference of the expected distance and $h$ is of the order $o(h)$. 
Moreover, compared to \cite{aku-aihp}, we allow for a larger class of limiting distributions. Indeed, our setting includes processes that cannot be characterized as the solution of an SDE, e.g., diffusions with sticky points.

\smallskip
There are further articles in the literature using random time grids to approximate a Markov process, under the additional assumption that it solves a one-dimensional SDE. 
In \cite{EL} the authors first fix a finite grid in the state space of the diffusion. Then they construct a Bernoulli random walk on this grid that
can be embedded into the diffusion. The authors determine the 
expected time for attaining one of the neighboring points by solving a PDE. 

\cite{milstein2015uniform} describes a similar approximation method for the 
Cox-Ingersoll-Ross (CIR) process. Also here the authors first fix a grid 
on $[0,\infty)$ and then construct a random walk on the grid that can 
be embedded into the CIR process. In contrast to \cite{EL}, the authors in 
\cite{milstein2015uniform} compute the distributions of the embedding stopping times (and 
not only their expected value) by solving a parabolic PDE. In the numerical 
implementation of the scheme the authors then draw the random time
increments from these distributions and thereby obtain a scheme that is
exact along a sequence of stopping times.
Note that in contrast to \cite{EL} and \cite{milstein2015uniform}, in our approach 
the space grid is not fixed a priori. Instead, we approximately fix the expected time lag between the consecutive embedding stopping times.
%The state distribution of the approximating Markov chain
%is then determined endogenously through a state-dependent scale factor (see $a_N$ defined in~\eqref{sf absorbing case} below).

Yet a further scheme that uses a random time partition to approximate a
diffusion $Y$ with discontinuous coefficients is suggested in \cite{LLP2019}. 
In contrast to our approach the distribution of the time increments is fixed there.
More precisely,
the authors of \cite{LLP2019} use the fact that the distribution of $Y$ sampled at an independent
exponential random time is given by the resolvent of the process. 
Consequently, if it is possible to generate random variables 
distributed according to the resolvent kernel,
one obtains an exact simulation of $Y$ at an exponentially distributed time.
Iterating this procedure and letting the parameter
of the exponential distribution go to infinity provides an approximation of~$Y$.

We remark that embeddings along random time grids have been recently employed in \cite{GLL2} and \cite{GLL2020} in order to obtain convergence rates of (F)BSDE approximations driven by Bernoulli increments.

\smallskip
%Finally, recall that if a diffusion solves an SDE with some regular coefficients, then one can approximate it along deterministic time grids arbitrarily well.
Recall that, while approximating solutions of SDEs on deterministic time grids
usually employs Euler-type schemes~\eqref{eq:def_X_intro} with Gaussian increments $(\xi_k)_{k\in\bbN}$,
we use Bernoulli increments in our paper.
%In this connection, we would like to mention that convergence results along equidistant time grids,
%including approximations by some Markov chains with Bernoulli increments,
%can be found, e.g., in Section~9.7 of \cite{KloedenPlaten}.
%\red{[and in Section~7.4 of \cite{EthierKurtz}?]}
In this connection, we would like to mention that, from the numerical perspective,
convergence results along equidistant time grids,
including approximations by the weak Euler schemes with Bernoulli increments,
can be found, e.g., in Section~14.1 of \cite{KloedenPlaten}.
From the more theoretical perspective, we refer to
Theorem~7.4.1 in \cite{EthierKurtz} and Theorem~IX.4.8 in \cite{JacodShiryaev},
which are some general functional limit theorems of the Trotter-Kato type for approximating diffusions.
A discussion of how to approximate controlled diffusions by Markov chains with Bernoulli increments can be found in \cite{KushnerDupuis}.
Another perspective on schemes with Bernoulli increments is suggested in
\cite{BSWOHRKK} and \cite{GHMR:19},
where, on certain machines (like field programmable gate arrays),
such schemes are shown to be more efficient for simulation algorithms.

While in our paper a continuous-time Markov process is approximated via (linearly-interpolated) \emph{discrete-time} Markov chains,
there is an alternative approach, pioneered in \cite{Stone63},
where the approximating processes are themselves \emph{continuous-time} Markov processes.
For a recent account, see \cite{BouVanden2018} and references therein.
A generalization of the latter approach for variable-speed random walks on trees contained in \cite{ALW:17} is worth mentioning as well.

\bigskip
The article is organized as follows. In Section~\ref{sec:MC approx} we rigorously formulate and discuss the functional limit theorem. In Section~\ref{sec:sdes} we discuss some of its implications for diffusions that can be described as solution of SDEs. In Sections \ref{sec 2} and~\ref{sec:higher moments} we explain, for a given general diffusion, how to embed an approximating coin tossing Markov chain into the diffusion and prove some properties of the embedding stopping times. Section~\ref{sec fctal lim thm} provides the proof of the functional limit theorem, where we, in particular, need the material discussed in Sections \ref{sec 2} and~\ref{sec:higher moments}. The functional limit theorem is shown under the additional assumption that if a boundary point is attainable, then it is absorbing. In Section~\ref{sec:reflecting} we explain how one can extend the functional limit theorem to general diffusions with reflecting boundary points. In the last two sections we illustrate our main result with diffusions exhibiting some stickiness. In Section~\ref{sec:examples} we construct coin tossing Markov chains approximating sticky Brownian motion, with and without reflection, respectively. In Section~\ref{sec:cantor_bm} we first describe a Brownian motion that is slowed down on the Cantor set, and secondly we explicitly construct coin tossing Markov chains that approximate this process arbitrarily well.

\section{Approximating general diffusions with Markov chains}\label{sec:MC approx}

Let $(\Omega, \cF, (\cF_t)_{t \ge 0}, (P_y)_{y \in I}, (Y_t)_{t \ge 0})$ be a one-dimensional continuous strong Markov process in the sense of Section~VII.3 in \cite{RY}. We refer to this class of processes as {\itshape general diffusions} in the sequel. We assume that the state space is an open, half-open or closed interval $I \subseteq \R$. We denote by $I^\circ=(l,r)$ the interior of $I$, where $-\infty\leq l<r\leq \infty$, and we set $\ol I=[l,r]$.
Recall that by the definition we have $P_y[Y_0=y]=1$ for all $y\in I$. 
We further assume that $Y$ is regular. This means that for every $y\in I^\circ$ and $x\in I$ we have that $P_y[H_x(Y)<\infty]>0$, where $H_x(Y)=\inf\{t\geq 0: Y_t=x \}$.
If there is no ambiguity,
we simply write $H_x$ in place of $H_x(Y)$. Moreover, for $a<b$ in $\ol I$ we denote by $H_{a,b}=H_{a,b}(Y)$
the first exit time of $Y$ from $(a,b)$,
i.e.\ $H_{a,b} = H_a\wedge H_b$.
Without loss of generality
we suppose that the diffusion $Y$ is in natural scale. If $Y$ is not in natural scale, then there exists a strictly increasing continuous function $s:I \to \R$, the so-called scale function, such that $s(Y_t)$, $t\geq 0$, is in natural scale. 
Let $m$ be the speed measure of the Markov process $Y$
(see VII.3.7 and~VII.3.10 in \cite{RY}).\footnote{\label{fn:conv_sm}There are different conventions concerning the normalization of the speed measure. We follow the convention of \cite{RY} and \cite{BS2002} and note that our speed measure is thus twice as large as the one for example found in \cite{RogersWilliams}.}
Recall that for all $a<b$ in $I^\circ$ we have
\begin{equation}\label{eq:06072018a1}
0<m([a,b])<\infty.
\end{equation}
Finally,
% in Sections~\ref{sec:MC approx}--\ref{sec:path_poly_gr} 
we also assume that if a boundary point is accessible, then it is absorbing. We drop this assumption in Section~\ref{sec:reflecting}, where we extend our approximation method to Markov processes with reflecting boundaries.
The extension works for both instantaneous and slow reflection.

Let $\ol h \in (0,\infty)$ and suppose that for every $h \in (0, \ol h)$ we are given a measurable function $a_{h}\colon \ol I \to [0,\infty)$ such that $a_h(l)=a_h(r)=0$ and for all $y\in I^\circ$ we have $y\pm a_h(y)\in I$. We refer to each function $a_h$ as a scale factor.
We next construct a sequence of Markov chains associated to the family of scale factors $(a_h)_{h\in (0, \ol h)}$. To this end
%For the rest of this section
we fix a starting point $y \in I^\circ$ of $Y$.
Let $(\xi_k)_{k \in \bbN}$ be an iid sequence of random variables,
on a probability space with a measure $P$,
satisfying $P(\xi_k = \pm 1) = \frac12$. 
%For every $h\in [0,\ol h]$ and $k\in \N_0$ let $t^{h}_k=kh$. 
We denote by $(X^h_{kh})_{k \in \bbN_0}$ the Markov chain defined by
\begin{equation}\label{eq:def_X}
X^h_0 = y
\quad\text{and}\quad
X^h_{(k+1)h} = X^h_{kh}+ a_h(X^h_{kh}) \xi_{k+1}, \quad \text{ for } k \in \bbN_0. 
\end{equation}
We extend $(X^h_{kh})_{k \in \bbN_0}$ to a continuous-time process by linear interpolation, i.e., for all $t\in[0,\infty)$, we set
\begin{equation}\label{eq:13112017a1}
X^h_t = X^h_{\lfloor t/h \rfloor h} + (t/h - \lfloor t/h \rfloor) (X^h_{(\lfloor t/h \rfloor +1)h} - X^h_{\lfloor t/h \rfloor h}). 
\end{equation}
To highlight the dependence of $X^h=(X^h_t)_{t\in[0,\infty)}$ on the starting point $y\in I^\circ$ we also sometimes write~$X^{h,y}$.

%We next define for all $y \in I^\circ$ and $a \in(0,\infty)$ with $a+y$ and $a-y \in I$
%\begin{align*}
%f(y,a) = 1_{(l,\infty)}(y-a)1_{(-\infty,r)}(y+a) \vee 1_{(l_h,r_h)}(y). 
%\end{align*}
%Notice that $f(y,a) = 1$ if $l$ and $r$ are inaccessible. 

To formulate our main result we need the following condition.

\smallskip\noindent
\textbf{Condition~(A)}
For all compact subsets $K$ of $I^\circ$ it holds that
\begin{equation}
\sup_{y\in K }  \left|
\frac{1}{2}\int_{(y-a_h(y),y+a_h(y))} (a_h(y)-|u-y|)\,m(du)
-h
\right|
\in o(h), \quad h \to 0.
\end{equation}

%We have the following result.

\begin{theo}\label{main thm sec 1}
Assume that Condition~(A) is satisfied. Then, for any $y\in I^\circ$,
the distributions of the processes
$(X^{h,y}_{t})_{t \in [0,\infty)}$ under $P$
converge weakly to the distribution of
$(Y_t)_{t \in [0,\infty)}$ under $P_{y}$, as $h \to 0$;
i.e., for every bounded and continuous functional\footnote{As usual,
we equip $C([0,\infty),\bbR)$ with the topology
of uniform convergence on compact intervals,
which is generated, e.g., by the metric
$$
d(x,y)=\sum_{n=1}^\infty 2^{-n}
\left(\|x-y\|_{C[0,n]}\wedge1\right),
\quad x,y\in C([0,\infty),\bbR),
$$
where $\|\cdot\|_{C[0,n]}$ denotes the sup norm
in $C([0,n],\bbR)$.}
$F\colon C([0,\infty),\bbR)\to\bbR$,
it holds that
\ben\label{eq:13112017a2}
E[F(X^{h,y})]\to E_y[F(Y)], \quad h\to 0.
\een
\end{theo}

\begin{remark}\label{rem:03012020a1}
To better explain Condition~(A), for every $\alpha<\beta$ in $I$, we introduce the Green function $G_{\alpha,\beta}\colon[\alpha,\beta]^2\to\bbR$ of $Y$ by the formula
$$
G_{\alpha,\beta}(u,v)=\frac{(\beta-u\vee v)(u\wedge v-\alpha)}{\beta-\alpha}, \quad u,v\in [\alpha,\beta]
$$
(recall that $Y$ is in natural scale) and observe that, for all $y\in[\alpha,\beta]$,
\begin{equation}\label{eq:03012020a1}
E_y[H_{\alpha,\beta}(Y)]=\int_{(\alpha,\beta)}G_{\alpha,\beta}(y,u)\,m(du)
\end{equation}
(see, e.g., Section~VII.3 in \cite{RY}).
It follows that, for any $y\in I^\circ$ and $a>0$ such that $y\pm a\in I$, it holds
\begin{align}
&E_y[H_{y-a,y+a}(Y)]
\notag\\
&=\int_{(y-a,y)}\frac12 (u-y+a)\,m(du)
+\int_{\{y\}}\frac12 a\,m(du)
+\int_{(y,y+a)}\frac12 (y+a-u)\,m(du)
\notag\\
&=\frac12\int_{(y-a,y+a)}(a-|u-y|)\,m(du).
\label{eq:03012020a3}
\end{align}
Thus, Condition~(A) is an analytic condition that is equivalent to requiring that the scale factors $(a_h)_{h\in(0,\ol h)}$ satisfy
$$
\sup_{y\in K }  \left|
E_y[H_{y-a_h(y),y+a_h(y)}(Y)]-h
\right|
\in o(h), \quad h \to 0,
$$
for any compact subset $K$ of $I^\circ$.
\end{remark}

\begin{remark}
It is worth noting that Condition~(A) is, in fact, nearly necessary
for weak convergence~\eqref{eq:13112017a2}
(see Example~\ref{ex:14122018a1}).
\end{remark}

\begin{remark}
For all $y\in I^\circ$, $h\in (0,\ol h)$ it holds that
$$
\int_{(y-a_h(y),y+a_h(y))} (a_h(y)-|u-y|)\,m(du)
=
\int_{I} (a_h(y)-|u-y|)^+\,m(du).
$$
This yields an alternative representation of Condition (A) which is occasionally used below.
\end{remark}

It is important to note that for every speed measure $m$ there exists a family of scale factors such that Condition~(A) is satisfied and hence every general diffusion $Y$ can be approximated
by Markov chains of the form~\eqref{eq:def_X}. Indeed, for all $y\in I^\circ$, $h\in (0,\ol h)$ let $\wh a_h(l)=\wh a_h(r)=0$ and
\ben\label{sf absorbing case}
\wh a_h(y) = \sup\left\{a \ge 0: y\pm a \in I \text{ and } \frac{1}{2}\int_{(y-a,y+a)} (a-|z-y|)\,m(dz) \le h\right\}
\een
and denote by $(\wh X^h)_{h\in (0,\ol h)}$ the associated family of processes defined in \eqref{eq:def_X} and \eqref{eq:13112017a1}. 
Then the proof of Corollary~\ref{cor:emcel} below shows that for all compact subsets $K$ of $I^\circ$ there exists $h_0 \in (0,\ol h)$ such that for all $y\in K$, $h\in (0,h_0)$ it holds that
$$
\frac{1}{2}\int_{(y-\wh a_h(y),y+\wh a_h(y))} (\wh a_h(y)-|z-y|)\,m(dz) = h.
$$
In particular, the family $(\wh a_h)_{h\in (0,\ol h)}$ satisfies Condition~(A) and 
we show in Section~\ref{sec 2} below that the Markov chain $(\wh X^h_{kh})_{k\in \N_0}$ is embeddable into $Y$ with a sequence of stopping times with expected time lag $h$. We refer to $(\wh X^{h}_{t})_{t \in [0,\infty)}$, $h\in (0,\ol h)$, as EMCEL approximations and write shortly $(\wh X^h_{kh})_{k\in \N_0} \in \text{EMCEL}(h)$.

\begin{corollary}\label{cor:emcel}
For every $y\in I^\circ$
the distributions of the EMCEL approximations
$(\wh X^{h,y}_{t})_{t \in [0,\infty)}$ under $P$
converge weakly to the distribution of
$(Y_t)_{t \in [0,\infty)}$ under $P_{y}$ as $h \to 0$.
\end{corollary}
\begin{proof}
Let $K$ be a compact subset of $I^\circ$. Without loss of generality assume that
$K=[l_0,r_0]$ with $l<l_0<r_0<r$. Let $a_0=\frac{r-r_0}{2}\wedge\frac{l_0-l}{2}\wedge 1$. It follows with dominated convergence that the function
$$
K\ni y\mapsto \frac{1}{2}\int_{I} (a_0-|u-y|)^+\,m(du)\in (0,\infty)
$$
is continuous. In particular, it is bounded away from zero, i.e., there exists $h_0 \in (0,\ol h)$ such that for all $y\in K$ it holds that
$\frac{1}{2}\int_{I} (a_0-|u-y|)^+\,m(du)\ge h_0$.
Next observe that for all $y\in K$ the function
$[0,a_0]\ni a \mapsto \frac{1}{2}\int_{I} (a-|u-y|)^+\,m(du) \in [0,\infty)$ is 
continuous and
strictly increasing. Hence for all $y\in K$, $h\in (0,h_0)$ the supremum in~\eqref{sf absorbing case} is a maximum and it holds that
$
\frac{1}{2}\int_{(y-\wh a_h(y),y+\wh a_h(y))} (\wh a_h(y)-|u-y|)\,m(du) = h.
$
In particular, Condition~(A) is satisfied and the statement of Corollary~\ref{cor:emcel} follows from Theorem~\ref{main thm sec 1}.
\end{proof}

%\begin{remark}
%In our old projects we determine for every $h\in (0,\ol h]$, $y\in I^\circ$ the scale factor $a_h(y)$ by solving the following equation in $a$
%\begin{equation}\label{eq:exact_sf}
%\int_{(y-a,y+a)} (a-|r-y|)\,m(dr)=2h.
%\end{equation}
%Obviously Condition~(A) is satisfied in this case. In general one can solve \eqref{eq:exact_sf} only approximately. Theorem~\ref{main thm sec 1} states that it suffices to solve \eqref{eq:exact_sf} approximately in such a way that
%\begin{equation}\label{eq:approx_sf}
%\left|\int_{(y-a,y+a)} (a-|r-y|)\,m(dr)-2h\right| \in o(h), \quad h\to 0,
%\end{equation}
%uniformly in $y\in I^\circ$.
%\end{remark}

\section{Application to SDEs}\label{sec:sdes}
A particular case of our setting is the case,
where $Y$ is a solution to the driftless SDE
\begin{equation}\label{eq:27092018a1}
dY_t=\eta(Y_t)\,dW_t,
\end{equation}
where $\eta\colon I^\circ\to\bbR$ is a Borel function
satisfying the Engelbert-Schmidt conditions
\begin{gather}
\eta(x)\ne0\;\;\forall x\in I^\circ,
\label{eq:27092018a2}\\[1mm]
\eta^{-2}\in L^1_{\loc}(I^\circ)
\label{eq:27092018a3}
\end{gather}
($L^1_{\loc}(I^\circ)$ denotes the set of Borel functions
locally integrable on~$I^\circ$).
Under \eqref{eq:27092018a2}--\eqref{eq:27092018a3}
SDE~\eqref{eq:27092018a1}
has a unique in law weak solution
(see \cite{ES1985} or Theorem~5.5.7 in \cite{KS}).
This means that there exists a pair of processes $(Y,W)$ on a filtered probability space $(\Omega, \cF, (\cF_t), P)$, with $(\cF_t)$ satisfying the usual conditions, such that $W$ is an $(\cF_t)$-Brownian motion and $(Y,W)$ satisfies SDE~\eqref{eq:27092018a1}.
The process $Y$ possibly reaches the endpoints $l$ or $r$ in finite time.
By convention we force $Y$ to stay in $l$ (resp.,~$r$) in this case.
%if it reaches $l$ (resp.,~$r$) in finite time.
This can be enforced in~\eqref{eq:27092018a1}
by extending $\eta$ to $\ol I$ with $\eta(l)=\eta(r)=0$.
In this example $Y$ is a regular continuous strong Markov
process with the state space being the interval
with the endpoints $l$ and~$r$
(whether $l$ and $r$ belong to the state space
is determined by the behavior of $\eta$ near the boundaries).
Moreover, $Y$ is in natural scale, and its speed measure
on $I^\circ$ is given by the formula
$$
m(dx)=\frac 2{\eta^2(x)}\,dx.
$$
In this situation a change of variables shows that it holds for 
all $h\in (0,\ol h)$, $y\in I^\circ$ that
\begin{equation}
\int_{(y-a_h(y),y+a_h(y))} (a_h(y)-|u-y|)\,m(du)
=
2a_h(y)^2\int_{-1}^1 \frac{1-|z|}{\eta^2(y+a_h(y)z)}\, dz.
\end{equation}
Condition~(A) hence becomes that for every compact subset $K$ of $I^\circ$ it holds that
\begin{equation}\label{eq:cond_A_sde}
\lim_{h\to 0}
\left(
\sup_{y\in K } 
\left|
\frac{a_h(y)^2}{h}\int_{-1}^1 \frac{1-|z|}{\eta^2(y+a_h(y)z)}\, dz
-1
\right|
\right)
=0.
\end{equation}

\begin{ex}[Brownian motion]\label{ex:14122018a1}
In the special case where $Y=W$ is a Brownian motion
(i.e., $I=\bbR$, $\eta(x)\equiv1$), Condition~(A) requires that for all compact sets $K\subset \R$ it holds that
$
\sup_{y\in K } 
\left|
\frac{a_h(y)^2}{h}
-1
\right|
\to
0
$ as $h\to 0$. In particular, Condition~(A) is satisfied for the choice $a_h(y)=\sqrt{h}$, $h\in (0,\infty)$, $y\in \R$, and we recover from Theorem~\ref{main thm sec 1} Donsker's functional limit theorem for the scaled simple random walk.

Moreover, in the case of a Brownian motion it is natural to restrict ourselves to space-homogeneous (i.e., constant) scale factors $a_h(y)\equiv a_h$, $h\in(0,\ol h)$, so that Condition~(A) takes the form $\lim_{h\to0}\frac{a_h^2}h=1$.
It is straightforward to show that the latter condition is also necessary for the weak convergence of approximations
\eqref{eq:def_X}--\eqref{eq:13112017a1}
driven by space-homogeneous scale factors
to the Brownian motion.
\end{ex}

\begin{ex}[Geometric Brownian motion]
Let $\sigma>0$ and assume that $\eta$ satisfies for all $x\in (0,\infty)$ that $\eta(x)=\sigma x$. Then the solution $Y$ of \eqref{eq:27092018a1} with positive initial value $Y_0=y\in (0,\infty)$ is a geometric Brownian motion. Its state space is $I=(0,\infty)$ and both boundary points are inaccessible. Note that for all
$y \in (0,\infty)$, $a\in (0,y)$ it holds that
$$
\int_{-1}^1 \frac{1-|z|}{\eta^2(y+az)}\, dz
=\frac{1}{(\sigma y)^2}\int_{-1}^1 \frac{1-|z|}{(1+az/y)^2}\, dz
= -\frac{1}{(\sigma a)^2} \log\left(1-\frac{a^2}{y^2}\right).
$$
Hence, Condition~(A) requires that for all compact sets $K\subset (0,\infty)$ it holds that
\begin{equation}\label{eq:cond_A_gbm}
\lim_{h\to 0}
\left(
\sup_{y\in K } 
\left|
\frac{1}{h \sigma^2}\log\left(1-\frac{a_h(y)^2}{y^2}\right)
+1
\right|
\right)
=0.
\end{equation}
To obtain the EMCEL approximation of $Y$ we solve for all $y\in (0,\infty)$, $h\in (0,\infty)$ the equation 
$
\frac{1}{h \sigma^2}\log\left(1-\frac{a^2}{y^2}\right)
+1=0
$
in $a$ and obtain $\wh a_h(y)=y \sqrt{1-e^{-\sigma^2 h}}$. Note that also the usual choice $a_h(y)=\sqrt{h}\sigma y$, $y\in (0,\infty)$, $h\in (0, 1/\sigma^2)$,
which corresponds to the weak Euler scheme
for geometric Brownian motion,
satisfies~\eqref{eq:cond_A_gbm}. 
\end{ex}

\subsection*{Convergence of the weak Euler scheme}
Throughout this subsection we assume that $I=\R$.
A common method to approximate solutions of SDEs is the Euler scheme. For equations of the form \eqref{eq:27092018a1} with initial condition $Y_0=y$ the Euler scheme $(X^{Eu,h}_{kh})_{k\in \N_0}$ with time step $h\in (0,\infty)$ is given by
$$
X^{Eu,h}_0=y
\quad\text{and}\quad
X^{Eu,h}_{(k+1)h} = X^{Eu,h}_{kh}+ \eta(X^{Eu,h}_{kh}) (W_{(k+1)h}-W_{kh}), \quad \text{ for } k \in \bbN_0. 
$$
Weak Euler schemes are variations of the Euler scheme, where the normal increments $W_{(k+1)h}-W_{kh}$, $k\in \N_0$, are replaced by an iid sequence of centered random variables with variance $h$. Therefore, with the choice $a_h(y)=\sqrt{h} \eta(y)$, $h\in (0,\infty)$, $y\in \R$, the Markov chain $(X^h_{kh})_{k\in \N_0}$ defined in \eqref{eq:def_X} represents a weak Euler scheme with Rademacher increments.

In this subsection we show how Theorem~\ref{main thm sec 1} can be used to derive new convergence results for weak Euler schemes. To this end let the setting of Section~\ref{sec:sdes} be given and let $a_h(y)=\sqrt{h} \eta(y)$, $h\in (0,\infty)$, $y\in \R$. 
Then it follows from \eqref{eq:cond_A_sde} that Condition~(A) is equivalent to assuming that for every compact subset $K \subset \R$ we have
\begin{equation}\label{eq:cond_A_sde_euler}
\sup_{y\in K } 
\left|
\int_{-1}^1 \frac{\eta^2(y) (1-|z|)}{\eta^2(y+\sqrt{h}\eta (y)z)}\, dz
-1
\right|
=
\sup_{y\in K } 
\left|
 \int_{-1}^1 
\frac{\eta^2(y)-\eta^2(y+\sqrt{h}\eta (y)z)}{\eta^2(y+\sqrt{h}\eta (y)z)}
\left( 1-|z|\right)
\, dz
\right|
\to 0,
\end{equation}
as $h\to 0$.

Suppose that $\eta$ is continuous, let $K\subset \R$ be compact and let $\eps>0$. Then $\eta$ is bounded on $K$ and since every continuous function is uniformly continuous on compact sets, we obtain that there exists $h_0 \in (0,\infty)$ such that for all $h\in (0,h_0]$, $y\in K$, $z\in [-1,1]$ it holds that
$$
|\eta(y)-\eta(y+\sqrt{h}\eta (y)z)|\le  \eps.
$$
By \eqref{eq:27092018a2} and the continuity of $\eta$ the function $\eta^2$ is strictly bounded away from $0$ on every compact subset of $\R$ and hence we obtain that there exists $C\in [0,\infty)$ such that for all
$h\in (0,h_0]$ it holds that
$$
\sup_{y\in K } 
\left|
 \int_{-1}^1 
\frac{\eta^2(y)-\eta^2(y+\sqrt{h}\eta (y)z)}{\eta^2(y+\sqrt{h}\eta (y)z)}
\left( 1-|z|\right)
\, dz
\right| \le C \eps.
$$
It follows with \eqref{eq:cond_A_sde_euler} that Condition~(A) is satisfied. Therefore we obtain the following Corollary of Theorem~\ref{main thm sec 1}.
\begin{corollary}\label{cor:euler}
Assume the setting of Section~\ref{sec:sdes} with $I=\R$ and that $\eta$ is continuous. Let $a_h \colon \R \to \R$ satisfy $a_h(y)=\sqrt{h} \eta(y)$ for all $h\in (0,\infty)$, $y\in \R$. 
Then
for all 
%$T\in (0,\infty)$ and 
$y\in \R$
the distributions of the processes $(X^{h,y}_{t})_{t \in [0,\infty)}$ under $P$ converge weakly to the distribution of $(Y_t)_{t \in [0,\infty)}$ under $P_{y}$, as $h \to 0$.
\end{corollary}

\begin{remark}
Corollary~\ref{cor:euler} complements convergence results for the Euler scheme for example obtained in \cite{yan} and \cite{gyongy98}. Theorem 2.2 in \cite{yan} shows weak convergence of the Euler scheme if $\eta$ has at most linear growth and is discontinuous on a set of Lebesgue measure zero. Theorem 2.3 in \cite{gyongy98} establishes almost sure convergence of the Euler scheme if $\eta$ is locally Lipschitz continuous. Moreover, \cite{gyongy98} allows for a multidimensional setting and a drift coefficient. In contrast, Corollary~\ref{cor:euler} above applies to the {\itshape weak} Euler scheme and does not require linear growth or local Lipschitz continuity of $\eta$.
\end{remark}

\begin{remark}
As stated in Corollary~\ref{cor:emcel}, EMCEL approximations can be constructed
for \textit{every} general diffusion. In particular, they can be used in cases where $\eta$ is not continuous and where (weak) Euler schemes do not converge (see, e.g., Section 5.4 in \cite{aku-jmaa}). In Sections~\ref{sec:examples} and \ref{sec:cantor_bm} we consider further irregular examples.
\end{remark}

\section{Embedding the chains into the Markov process}\label{sec 2}
In this section we construct the embedding stopping times. 
Throughout the section we assume the setting of Section~\ref{sec:MC approx}.

We need to introduce an auxiliary subset of $I^\circ$.
To this end, if $l> -\infty$, we define, for all $h\in(0,\ol h)$,
\begin{align*}
l_h := l+ \inf\left\{ a \in \left(0,\frac{r-l}2\right]:
a<\infty
\;\;\text{and}\;\;
\frac12\int_{(l, l+2a)} (a - |u-(l+a)|)\,m(du) \ge h \right\}, 
\end{align*}
where we use the convention $\inf \emptyset = \infty$.
If $l = -\infty$, we set $l_h = -\infty$. 
Similarly, if $r < \infty$, then we define, for all $h\in(0,\ol h)$,
\begin{align*}
r_h := r- \inf\left\{ a \in \left(0,\frac{r-l}2\right]:
a<\infty
\;\;\text{and}\;\;
\frac12\int_{(r-2a, r)} (a - |u-(r-a)|)\,m(du) \ge h \right\}.
\end{align*}
If $r = \infty$, we set $r_h = \infty$.
It is worth noting that $l$ is inaccessible if and only if $l_h = l$ for all $h\in (0,\ol h)$;
and, similarly, $r$ is inaccessible if and only if $r_h = r$ for all $h\in (0,\ol h)$.
This is verified in Remark~\ref{rem:03012020a2} below.
The auxiliary subset is defined by
$$
I_h = (l_h,r_h)\cup\left\{y\in I^\circ \colon y\pm a_h(y)\in I^\circ\right\}.
$$

Now we have everything at hand to start constructing a sequence of embedding stopping times. Suppose $Y$ starts at a point $y\in I^\circ$ and fix $h \in(0, \ol h)$. Set $\tau^h_0=0$.
Let $\sigma^h_1 = H_{y-a_h(y), y+a_h(y)}$.
Recall that we have
$$
E_y[\sigma^h_1] = \frac12 \int_{(y-a_h(y),y+a_h(y))} (a_h(y)-|u-y|)\,m(du)
$$
(see Remark~\ref{rem:03012020a1}).
We now define $\tau^h_1$ by distinguishing two cases. 

\medskip\noindent
{\itshape Case 1:} $y \in I_h$  (i.e., $y\in (l_h,r_h)$ or $y\pm a_h(y)\in I^\circ$). In this case we set $\tau^h_1 = \sigma^h_1$.

\medskip\noindent
{\itshape Case 2:} $y \notin I_h$ (i.e., $y\notin (l_h,r_h)$ and ($y+a_h(y)=r$ or $y-a_h(y)= l$)). In this case we deterministically extend $\sigma^h_1$ so as to make it have expectation $h$.
Observe that by the definition of $l_h$ and $r_h$ we have in this case $E_y[\sigma^h_1] \le h$. 
Moreover, we can assume in this case that it must hold that $P_y(Y_{\sigma^N_1} \in \{l,r\})=\frac12$
(only in the case $\max\{|l|,|r|\}<\infty$, $y=\frac{l+r}2$ and $a_h(y)=\frac{r-l}2$
this probability is $1$, but we exclude this case by considering a sufficiently small $h$, so that $a_h(\frac{l+r}2)<\frac{r-l}2$; notice that Condition (A) implies that, for any $y\in I^\circ$, $\lim_{h\to 0}a_h(y)=0$).  We define $\tau^h_1$ by  
\begin{align*}
\tau^h_1 = \sigma^h_1 + 2 \left(h  - E[\sigma^h_1]\right)
1_{ \{ l, r\}}(Y_{\sigma^h_1}). 
\end{align*}
Observe that the definition implies $E_y[\tau^h_1] = h$ and that the three random variables $Y_{\tau^h_1}$, $Y_{\sigma^h_1}$ and $X^{h,y}_h$ have all the same law.  

\smallskip
We can proceed in a similar way to define the subsequent stopping times. Let $k\in \bbN$. Suppose that we have already constructed $\tau^h_k$. We first define $\sigma^h_{k+1} = \inf\{ t \ge \tau^h_k: |Y_t - Y_{\tau^h_k}| = a_h(Y_{\tau^h_k}) \}$. 
On the event $\{Y_{\tau^h_k} \in I_h\}$
%$\{Y_{\tau^h_k} \in [l_h, r_h] \} \cup \left(\{Y_{\tau^h_k} +a(Y_{\tau^h_k}) < r \} \cap \{Y_{\tau^h_k} -a(Y_{\tau^h_k}) > l \} \right)$ 
we set $\tau^h_{k+1} = \sigma^h_{k+1}$.
On the event $\{Y_{\tau^h_k} \notin I_h\}$
%$\{Y_{\tau^h_k} \notin [l_h, r_h] \} \cap \left(\{Y_{\tau^h_k} +a(Y_{\tau^h_k}) = r \} \cup \{Y_{\tau^h_k} -a(Y_{\tau^h_k}) = l \} \right) $ 
we extend $\sigma^h_{k+1}$ as follows. Note that $Y_{\tau^h_k}$ takes only finitely many values. Let $v \in I\setminus (l_h, r_h)$ be a possible value of $Y_{\tau^h_k}$ such that $v-a_h(v) = l$ or $v+a_h(v) = r$. Consider the event $A = \{Y_{\tau^h_k} = v \}$. Observe that $c:= E_y[\sigma^h_{k+1}-\tau^h_k | A] \le  h$.
%\red{[For EMCEL: yes. For a general scheme: not necessarily.]}
We extend $\sigma^h_{k+1}$ on the event $A$ by setting 
\begin{align}\label{eq:def_tau}
\tau^h_{k+1} = \sigma^h_{k+1} +
2\left(h  - c\right)
1_{  \{ l , r\}}(Y_{\sigma^h_{k+1}} )
\end{align}   
(notice that $P_y(Y_{\sigma^h_{k+1}} \in \{l,r\} | A)=\frac12$). 
This implies that $E_y[\tau^h_{k+1} - \tau^h_k| \cF_{\tau^h_k}] = h$ on the event $\{Y_{\tau^h_k} \notin I_h\}$. Moreover, the processes $(Y_{\tau^h_j})_{j\in \{0,\ldots,k+1\}}$ and $(X^{h,y}_{jh})_{j\in \{0,\ldots,k+1\}}$ have the same law. 
To sum up, we have the following. 

\begin{propo}\label{main thm sec 2}
For all $h\in (0,\ol h)$ and $y\in I^\circ$ the sequence of stopping times $(\tau^h_k)_{k \in \Znn}$ satisfies
\begin{enumerate}
\item 
$
\Law_{P_y}
\left(Y_{\tau^h_k}; k\in \Znn \right)
=\Law_P
\left(X^{h,y}_{kh}; k\in \Znn\right).
$
%$(Y_{\tau^h_k})_{k \in \Znn} \stackrel{d}{=} (X^h_{kh})_{k \in \Znn}$, 
\item For all $k\in \N_0$ we have
\begin{align*}
&E_y(\tau^h_{k+1} - \tau^h_k| \cF_{\tau^h_k}) \\
&=\left\{ \begin{array}{lc}
\frac{1}{2}\int_{(Y_{\tau^h_k}-a_h(Y_{\tau^h_k}),Y_{\tau^h_k}+a_h(Y_{\tau^h_k}))} (a_h(Y_{\tau^h_k})-|u-Y_{\tau^h_k}|)\,m(du), &\text{ if } Y_{\tau^h_k} \in I_h, \\
h, & \text{ if }Y_{\tau^h_k} \notin I_h.
\end{array}\right.
\end{align*} 
\end{enumerate}
\end{propo}

\section{Higher moment estimates for exit times}\label{sec:higher moments}
In this section we provide some moment estimates for the exit times of $Y$ from intervals. We use the estimates in the next section to prove convergence in probability of $\sup_{k\in \{1,\ldots, \lfloor T/h \rfloor \}}\left|\tau^h_k-kh\right|$ to zero, where $(\tau^h_k)_{k \in \Znn}$ is the sequence of embedding stopping times from Section \ref{sec 2}. This is a crucial ingredient in the proof of our main result, Theorem~\ref{main thm sec 1}.

We introduce the function $q\colon I^\circ\times\ol I \to [0,\infty]$ defined by
\begin{align}\label{eq:def_q}
q(y,x) = \frac12 m(\{y\}) |x-y| + \int_y^x m((y,u))du, 
\end{align}
where for $u<y$ we set $m((y,u)):= - m((u,y))$.
Notice that, for $y\in I^\circ$, the function $q(y,\cdot)$
is decreasing on $[l,y]$ and increasing on $[y,r]$.
For our analysis the key property of $q$ is that it makes
 the process $q(y,Y_t) - (t \wedge H_{l,r})$, $t \in [0,\infty)$,
a $P_y$-local martingale (see Lemma~\ref{propo q-t mart} below).
Moreover, $q$ plays a central role in Feller's test for explosions:
for any $y\in I^\circ$,
\begin{align}\label{fellerl}
l \text{ is accessible (i.e., }l\in I)
&\;\;\Longleftrightarrow\;\;
q(y,l) < \infty,\\[2mm]
\label{fellerr}
r \text{ is accessible  (i.e., }r\in I)
&\;\;\Longleftrightarrow\;\;
q(y,r) < \infty
\end{align}
(see, e.g., Lemma~2.1 in \cite{AKKK17} or Theorem~3.3 in \cite{AKKU2018}).
Consequently, $q$ is finite on $I^\circ\times I$.
Notice that for all $y,z \in I^\circ$
and $x\in I$ we have
\begin{align}\label{master formula}
q(z,x) = q(y,x) - q(y, z) - \frac{\partial^0 q}{\partial x}(y, z) (x-z), 
\end{align}
where $\frac{\partial^0 q}{\partial x}(y, x) = \frac12 (\frac{\partial^+ q}{\partial x}+\frac{\partial^- q}{\partial x})(y, x)$.

\medskip
The following lemma identifies a local martingale associated to $Y$.

\begin{lemma}\label{propo q-t mart}
Let $y\in I^\circ$.
Then the process $q(y,Y_t) - (t \wedge H_{l,r})$, $t \in [0,\infty)$,
is a $P_y$-local martingale. 
\end{lemma}

\begin{proof}
Let $a, b \in I$ with $a < y < b$. We first show that $q(y,Y_{t \wedge H_{a,b}}) - t \wedge H_{a,b}$, $t \in[0,\infty)$, is a $P_y$-martingale.
It follows from~\eqref{eq:03012020a1} that
\begin{align*}
E_y H_{a,b} & = \frac{1}{b-a} \left[\int_{(a,y)} (b-y)(x-a) m(dx)+ \int_{(y,b)} (b-x)(y-a) m(dx) + (b-y)(y-a)m(\{y\}) \right]
\\
& =\frac{b-y}{b-a} \left[\int_{(a,y)} (x-a) m(dx)+(y-a)\frac{m(\{y\})}{2} \right] + \frac{y-a}{b-a}\left[\int_{(y,b)} (b-x) m(dx) + (b-y)\frac{m(\{y\})}{2} \right]
\\
& = \frac{b-y}{b-a} q(y,a) + \frac{y-a}{b-a} q(y,b)
\\
& = E_y q(y, Y_{H_{a,b}}),
\end{align*}
where, in the third equality, we use the representation
$\int_{(a,y)}1_{\{u<x\}}\,du$ for $x-a$
(and the similar one for $b-x$)
and apply Fubini's theorem. Thus,
\begin{equation}\label{eq 1.step}
E_y H_{a,b} = E_y q(y, Y_{H_{a,b}}).
\end{equation}
Next, we observe that for all $t\in [0,\infty)$ it holds
\begin{equation}\label{aux eq 230716}
E_y\left[q(y,Y_{H_{a,b}}) - H_{a,b} | \cF_t \right] 
= (q(y,Y_{H_{a,b}}) - H_{a,b}) 1_{\{ H_{a,b} \le t \}} + E_y\left[q(y,Y_{H_{a,b}}) - H_{a,b} | \cF_t \right] 1_{\{ H_{a,b} > t \}}. 
\end{equation}
On the event $\{H_{a,b} > t\}$ we have $q(y,Y_{H_{a,b}}) - H_{a,b} = q(y,Y_{H_{a,b}})\circ \theta_t - H_{a,b}\circ \theta_t -t$,
where $\theta_t$ denotes the shift operator for~$Y$
(see Chapter~III in \cite{RY}).
The Markov property and~\eqref{eq 1.step} imply that on the event $\{ H_{a,b} > t \}$ we have $P_y$-a.s.
\begin{align}\label{051017_1}
E_y\left[q(y,Y_{H_{a,b}}) - H_{a,b} | \cF_t \right] = & E_{z} [q(y,Y_{H_{a,b}}) - H_{a,b}]\Big|_{z = Y_t} -t \nonumber \\
= & E_{z} [q(y,Y_{H_{a,b}}) - q(z,Y_{H_{a,b}})]\Big|_{z = Y_t}-t. 
\end{align}
Formula \eqref{master formula} yields for all $z\in I^\circ$
\begin{align*}
q(y,Y_{H_{a,b}}) - q(z,Y_{H_{a,b}}) = q(y,z) + \frac{\partial^0 q}{\partial x} (y,z) (Y_{H_{a,b}} - z).
\end{align*}
Since $E_z[Y_{H_{a,b}} - z] = 0$ for all $z\in I^\circ$, equation \eqref{051017_1} implies that on the event $\{ H_{a,b} > t \}$ we have $P_y$-a.s.
\begin{align*}
E_y\left[q(y,Y_{H_{a,b}}) - H_{a,b} | \cF_t \right] 
& = q(y,Y_t) - t. 
\end{align*}
Together with \eqref{aux eq 230716} this yields for all $t\in [0,\infty)$
\begin{align*}
E_y\left[q(y,Y_{H_{a,b}}) - H_{a,b} | \cF_t \right] 
= q(y,Y_{t \wedge H_{a,b}}) - t \wedge H_{a,b},
\end{align*}
which shows that $q(y,Y_{t \wedge H_{a,b}}) - t \wedge H_{a,b}$, $t \in[0,\infty)$, is a $P_y$-martingale.

The statement of the lemma follows via a localization argument. If $l \notin I$, then choose a decreasing sequence $(l_n)_{n\in \N}\subseteq I$ with $l_1<y$ and $\lim_{n \to \infty} l_n= l$. If $l \in I$, set $l_n = l$ for all $n\in \N$. Similarly, if $r \notin I$, then choose an increasing sequence $(r_n)_{n\in \N}\subseteq I$ with $r_1>y$ and $\lim_{n\to \infty} r_n = r$, and if $r \in I$, then set $r_n = r$ for all $n\in \N$. The sequence of stopping times $\inf\{t \ge 0\colon X_t \notin [l_n,r_n]\}$, $n\in \N$,
 is then a localizing sequence for the process $q(y,Y_t) - (t \wedge H_{l,r})$, $t\in [0,\infty)$.   
\end{proof}

\begin{remark}\label{rem:03012020a2}
We still owe verifying that, for $l_h$ and $r_h$ introduced in Section~\ref{sec 2},
$l$ is inaccessible if and only if $l_h = l$ for all $h\in (0,\ol h)$;
and $r$ is inaccessible if and only if $r_h = r$ for all $h\in (0,\ol h)$.
Let us prove this equivalence for the left boundary point.
To this end, the following identity is helpful:
for $y\in I^\circ$ and $a\in(0,\infty)$ such that $y\pm a\in\ol I$,
it holds
\begin{equation}\label{eq:03012020a2}
\int_{(y-a,y+a)} (a-|u-y|)\,m(du)=q(y,y+a)+q(y,y-a).
\end{equation}
Indeed, if $y\pm a\in I$, then \eqref{eq:03012020a2} follows from \eqref{eq:03012020a3} and~\eqref{eq 1.step}.
If we only have $y\pm a\in\ol I$, then \eqref{eq:03012020a2} follows by the monotone convergence argument.

Now, if $l=-\infty$, then the equivalence is clear.
Let $l>-\infty$. Consider a small enough $a>0$. For the integral in the definition of $l_h$, we have due to~\eqref{eq:03012020a2}
\begin{equation}\label{eq:03012020a4}
\int_{(l, l+2a)} (a - |u-(l+a)|)\,m(du)=q(l+a,l)+q(l+a,l+2a),
\end{equation}
where $l+a,l+2a\in I^\circ$, and hence the second term on the right-hand side is finite,
i.e., the integral is infinite if and only if the first term on the right-hans side is infinite.
We conclude by applying Feller's test~\eqref{fellerl}:
$l$ is inaccessible if and only if the integral in~\eqref{eq:03012020a4} is infinite for all small $a>0$,
and the latter holds if and only if $l_h=l$ for all $h\in(0,\ol h)$.
\end{remark}

The next result provides conditions guaranteeing
that moments of a stopping time $\tau$ can be bounded 
against an integral with respect to the distribution of $Y_\tau$. 
\begin{theo}\label{thm:mom_gen_st}
Let $\alpha \in [1,\infty)$ and let $y\in I$. Let $\tau$ be a stopping time 
such that $\tau \le H_{l,r}(Y)$, $P_y$-a.s.\ and the process $(q^\alpha(y,Y_{\tau\wedge t}))_{t\in [0,\infty)}$ is of class~(D) under $P_y$. 
Then $\tau<\infty$, $P_y$-a.s.\ and it holds that
\begin{equation}\label{eq:mom_gen_st}
E_y[\tau^\alpha]\le \alpha^\alpha E_y[q^\alpha(y,Y_{\tau})].
\end{equation}
\end{theo}
\begin{proof}
If $y\in  I\setminus I^\circ$, then $\tau=0$ and \eqref{eq:mom_gen_st} is satisfied. For the remainder of the proof we assume that $y\in I^\circ$. 
We first show by contradiction that $\tau<\infty$ $P_y$-a.s. So assume that
$P_y(\tau=\infty)>0$. Since on $\{\tau=\infty\}$ we necessarily have $\tau=H_{l,r}(Y)$, we
obtain 
$$
\limsup_{t\to \infty}q^\alpha(y,Y_{t\wedge \tau})=\infty \qquad P_y\text{-a.s.\ on } \{\tau=\infty\}.
$$
For any $n\in \N$ let $\sigma_n=\inf\{t\in [0,\infty): q^\alpha(y,Y_{\tau\wedge t})\ge n\}$.
Then we obtain that 
$$
\lim_{n\to \infty}q^\alpha(y,Y_{\sigma_n\wedge \tau})=\infty \qquad P_y\text{-a.s.\ on } \{\tau=\infty\},
$$
which contradicts the uniform integrability of $\{q^\alpha(y,Y_{\sigma_n\wedge \tau}) \}_{n\in \N}$. Thus, we proved that $\tau<\infty$ $P_y$-a.s. and, in particular, that $Y_\tau$ on the right-hand side of \eqref{eq:mom_gen_st} is well-defined.

According to Lemma~\ref{propo q-t mart} the process
$N_t := q(y,Y_t) - (t \wedge H_{l,r}(Y))$, $t\ge 0$, is a $P_y$-local martingale.
The product formula yields for all $t \in [0, H_{l,r}(Y)]$
\ben\label{eq:prod_form}
t^{\alpha-1}q(y,Y_t)&=& t^{\alpha-1}N_t + t^\alpha =(\alpha-1) \int_0^t s^{\alpha-2} N_s ds + \int_0^t s^{\alpha-1} dN_s + t^\alpha \nonumber \\
&=& (\alpha-1)\int_0^t s^{\alpha-2} q(y,Y_s) ds + \int_0^t s^{\alpha-1} dN_s +\frac{1}{\alpha} t^\alpha\nonumber \\
&\ge& \int_0^t s^{\alpha-1} dN_s + \frac1\alpha t^\alpha. \label{hanoi aux ineq 1}
\een
Note that $(\int_0^t s^{\alpha-1} dN_s)_{t\ge 0}$ is a local martingale and let $(\tau'_n)_{n\in \N}$ be a localizing sequence for it. Set $\tau_n:=n\wedge\tau'_n$ for all $n\in \N$. In particular, it holds $E_y[\tau_n^\alpha]<\infty$ for all $n\in \N$.
With inequality \eqref{hanoi aux ineq 1} and H\"older's inequality we obtain for all $n\in \N$ that
\be
E_y[(\tau_n\wedge\tau)^\alpha]\le \alpha E_y[(\tau_n\wedge\tau)^{\alpha-1}q(y,Y_{\tau_n\wedge\tau})]\le \alpha (E_y[(\tau_n\wedge\tau)^\alpha])^{\frac{\alpha-1}{\alpha}}(E_y[q^\alpha(y,Y_{\tau_n\wedge\tau})])^{\frac{1}{\alpha}}.
\ee
This implies for all $n\in \N$
\ben\label{eq:140914a12}
E_y[(\tau_n\wedge\tau)^\alpha]\le \alpha^\alpha E_y[q^\alpha(y,Y_{\tau_n\wedge\tau})].
\een
By monotone convergence the left-hand side converges to $E_y[\tau^\alpha]$ as $n\to\infty$. Since the process $(q^\alpha(y,Y_{\tau\wedge t}))_{t\in [0,\infty)}$ is of class (D), it follows that the family $(q^\alpha(y,Y_{\tau_n\wedge\tau}))_{n\in \N}$ is uniformly integrable. Vitali's convergence theorem implies that $E_y[q^\alpha(y,Y_{\tau_n\wedge\tau})] \to E_y[q^\alpha(y,Y_{\tau})]$ as $n\to \infty$. Therefore we obtain
\be
E_y[\tau^\alpha ]\le \alpha^\alpha E_y[q^\alpha (y,Y_\tau)],
\ee
which is precisely~\eqref{eq:140914a11}.
\end{proof}

\begin{remark}
Under the assumption of Theorem~\ref{thm:mom_gen_st} we have equality in~\eqref{eq:mom_gen_st} for $\alpha=1$, i.e., $E_y[\tau]=E_y[q(y,Y_\tau)]$.
Indeed, the inequality~$\le$ is provided by~\eqref{eq:mom_gen_st},
while, for the reverse inequality~$\ge$, use that equality holds in~\eqref{eq:prod_form} for $\alpha=1$,
localize~\eqref{eq:prod_form}, compute expectations of the both sides and apply Fatou's lemma.
\end{remark}

From Theorem~\ref{thm:mom_gen_st} we obtain the following
moment estimate for first exit times.

\begin{corollary}\label{2moment}
Let $\alpha \in [1,\infty)$, let $y\in I$ and let $a\in [0,\infty)$
%let $y\in I^\circ$ and let $a\in (0,\infty)$
 be such that $[y-a,y+a]\subseteq I$.
 Then it holds that
\ben\label{eq:140914a11}
E_y[(H_{y-a,y+a}(Y))^\alpha] \le \frac{\alpha^\alpha}{2}\left( q^\alpha(y,y-a)+q^\alpha(y,y+a)\right).
\een
\end{corollary}

\begin{proof}
Clearly, $H_{y-a,y+a}(Y)\le H_{l,r}(Y)$, $P_y$-a.s.
Moreover, under~$P_y$,
the process $(q^\alpha(y,Y_{H_{y-a,y+a}(Y)\wedge t}))_{t\in [0,\infty)}$ is bounded
(see \eqref{fellerl} and~\eqref{fellerr})
and hence of class~(D).
Inequality \eqref{eq:140914a11} then follows from Theorem~\ref{thm:mom_gen_st}.
\end{proof}

\section{Proof of Theorem~\protect\ref{main thm sec 1}}
\label{sec fctal lim thm}
In this section we prove Theorem~\ref{main thm sec 1} and show that under 
Condition~(A) the processes $(X^h)_{h\in (0,\ol h)}$ converge in distribution to $Y$.
We use the embedding stopping times $(\tau^h_k)_{k \in \Znn}$ constructed in Section~\ref{sec 2} 
and control the temporal errors $|\tau^h_k-kh|$, $h\in (0,\ol h)$, $k\in \N_0$. To this end, for every $h\in (0,\ol h)$ we apply the Doob decomposition to the process $(\tau^h_k-kh)_{k\in \N_0}$ and write $\tau^h_k-kh=M^h_k+A^h_k$, $k\in \N_0$, for a martingale $M^h$ and a predictable process $A^h$. 
Condition~(A) guarantees that $A^h$ converges to $0$ as $h\to 0$ 
(see Proposition~\ref{theo:conv_A} below). We show that also the martingale part $M^h$ can be nicely controlled (see Proposition~\ref{coro sec 3} below). %To ensure that also $M^h$ tends to $0$ as $h \to 0$ we impose the following condition .

%\textbf{Condition~(M)} 
%\begin{equation}
%\sup_{y\in I } \int_{(y-a_h(y),y+a_h(y))} (a_h(y)-|u-y|)\,m(du)
%\in o(\sqrt{h}), \quad h\to 0.
%\end{equation}
%Note that Condition~(A) implies Condition~(M).

For all $h\in (0,\ol h)$ let $(\tau^h_k)_{k\in \N}$ be the sequence of embedding stopping times defined in Section~\ref{sec 2}. Then we have the following result about the time lags $\rho^h_k=\tau^h_{k}-\tau^h_{k-1}$, $h\in (0,\ol h)$, $k\in \N$, between consecutive embedding stopping times.
%Moreover, let $\rho^N_k = \tau^N_{k}-\tau^N_{k-1}$, $k \in \N$.

\begin{lemma}\label{l:st_ui}
Let $\alpha \in [1,\infty)$, $h\in (0,\ol h)$ and $y \in I^\circ$.
Then it holds that
\begin{equation}
\sup_{k \in \N }\left \| \rho^h_k
\right \|_{L^\alpha(P_y)}
\le 2h+
2^{1-1/\alpha} \alpha \sup_{z\in I}\left( E_z[H_{z-a_h(z), z + a_h(z)}(Y)]\right).
\end{equation}
\end{lemma}

\begin{proof}
By construction of the sequence $(\tau_k^h)_{k\in \N_0}$
(in particular, recall~\eqref{eq:def_tau}) it holds for all $k\in \N$ that
$$
\rho_k^h=\tau^h_{k}-\tau^h_{k-1}\le \inf\left\{t\ge \tau^h_{k-1}\colon |Y_t-Y_{\tau^h_{k-1}}|=a_h(Y_{\tau^h_{k-1}})\right\}+2h.
$$
This and the strong Markov property of $Y$ imply for all $k\in \N$ that
\begin{equation}\label{eq:def_tau_bound}
\begin{split}
E_y\left[\left(\tau^h_{k}-\tau^h_{k-1}\right)^\alpha \right]
%&=
%E_y\left[E_{Y_{\tau^h_{k-1}}}\left[\left(\tau^h_{k}-\tau^h_{k-1}\right)^\alpha \right] \right]\\
&\le 
E_y\left[ E_{z}\left[\left(H_{z-a_h(z),z+a_h(z)}(Y)+2h\right)^\alpha \right] \Big |_{z= Y_{\tau^h_{k-1}}}\right]\\
&\le \sup_{z\in I}E_{z}\left[\left(H_{z-a_h(z),z+a_h(z)}(Y)+2h\right)^\alpha \right].
\end{split}
\end{equation}
Notice that
\begin{equation}\label{eq:24102017a1}
a^\alpha+b^\alpha\le (a+b)^\alpha \quad \text{for } a,b \in [0,\infty),
\end{equation}
because the function $x\mapsto x^\alpha$, $x\in [0,\infty)$, is convex, increasing and starts in zero.
It follows from the triangle inequality, Corollary~\ref{2moment} and \eqref{eq:24102017a1}  that for all $z\in I$ it holds 
\be 
%\|\tau\|_{L^\alpha(P_z)}
\|H_{z-a_h(z),z+a_h(z)}(Y)+2h\|_{L^\alpha(P_z)}
&\le  & \left(E_z[H_{z-a_h(z), z + a_h(z)}(Y)^\alpha]\right)^{\frac{1}{\alpha}}+2h\\
& \le & 2^{-1/\alpha} \alpha \left( q^\alpha(z,z-a_h(z))+q^\alpha(z,z+a_h(z))\right)^{\frac{1}{\alpha}}+2h\\
& \le & 2^{-1/\alpha} \alpha\left( q(z,z-a_h(z))+q(z,z+a_h(z))\right)+2h\\
&=&2^{1-1/\alpha} \alpha E_z[H_{z-a_h(z), z + a_h(z)}(Y)]+2h.
%&=&
%2^{-1/\alpha} \alpha\left(\int_{(y-a_h(y),y+a_h(y))} (a_h(y)-|r-y|)\,m(dr) \right)+2h.
\ee
%Condition~(M) ensures that
%\begin{equation}
%\limsup_{h\to 0}
%\frac{\sup_{y\in I^\circ }\|\tau\|_{L^\alpha(P_y)}}{\sqrt{h}}
%\le 2^{-1/\alpha}\alpha
%\limsup_{h\to 0}
%\frac{\sup_{y\in I^\circ} \left(\int_{(y-a_h(y),y+a_h(y))} (a_h(y)-|r-y|)\,m(dr) \right)}{\sqrt{h}}
%=0
%\end{equation}
%while Condition (M$\lambda$) ensures that there exists $C\in (0,\infty)$ (only depending on $\alpha$ and $\lambda$) such that
%\begin{equation}
%\sup_{y\in I^\circ }\|\tau\|_{L^\alpha(P_y)}\le C h^{\lambda+\frac{1}{2}}.
%\end{equation}
Combining this with \eqref{eq:def_tau_bound} completes the proof.
%Let $N\in \N$ and define $G(y,a_N(y)):= \frac{1}{2}(q(y,y+a_N(y)) + q(y,y-a_N(y))$. 
%\noindent
%\underline{First case:} $G(y,a_N(y))=\frac{1}{N}$.  Then it follows from Lemma \ref{2moment} and the nonnegativity of $q$ that
%\be 
%E_y[\tau^2]&\le & 2\left( q^2(y,y-a_N(y))+q^2(y,y+a_N(y))\right)\\
%&\le & 2 \left( q(y,y-a_N(y))+q(y,y+a_N(y))\right)^2\\
%&= & 2(2G(y,a_N(y)))^2= \frac{8}{N^2}.
%\ee
%\noindent
%\underline{Second case:} $G(y,a_N(y))<\frac{1}{N}$. In this case at least one boundary is finite. Here we consider the case $l>-\infty$, $r=\infty$. It follows from $G(y,a_N(y))<\frac{1}{N}$ that $a_N(y)=y-l$ and $q(y,l+)<\infty$. Moreover, by construction, it holds that 
%$$\tau=H_{l,2y-l}(Y)+2\left(\frac{1}{N}-G(y,a_N(y)\right)1_{\{Y_{H_{l,2y-l}}=l\}}\le H_{l,2y-l}(Y)+2\left(\frac{1}{N}-G(y,a_N(y)\right).$$
%Then it follows again from Lemma \ref{2moment} that
%\be 
%E_y[\tau^2]&\le & 2E_y[H_{l,2y-l}(Y)^2]+8\left(\frac{1}{N}-G(y,a_N(y)\right)^2\\
%& \le & 4\left( q^2(y,l+)+q^2(y,2y-l)\right)+\frac{8}{N^2}\\
%& \le & 4\left( q(y,l+)+q(y,2y-l)\right)^2+\frac{8}{N^2}\\
%&= & 4(2G(y,a_N(y)))^2+\frac{8}{N^2}\le  \frac{24}{N^2}.
%\ee
%The remaining cases follow by similar arguments.
\end{proof}
Below, for a random variable $\xi$,
it is convenient to use the notation
$$
\|\xi\|_{L^\alpha(P_y)}=\left(
E_y|\xi|^\alpha\right)^{1/\alpha}
$$
for all $\alpha\in(0,\infty)$, even though it is not a norm
for $\alpha\in(0,1)$. Notice that
$\|\xi\|_{L^\alpha(P_y)}\le\|\xi\|_{L^\beta(P_y)}$
for $0<\alpha<\beta$ by the Jensen inequality.

\begin{propo}\label{coro sec 3}
Let $\alpha \in (0,\infty)$.
Then there exists a constant $C(\alpha)\in (0,\infty)$
such that, for all $T\in(0,\infty)$, $y \in I^\circ$ and $h\in (0,\ol h)$,
it holds that
\begin{multline}\label{eq:lalphamart}
\left\|\sup_{k\in \{1,\ldots, \lfloor T/h\rfloor \}}\left|\tau^h_k-\sum_{n=1}^k E[\rho^h_n|\mathcal F_{\tau^h_{n-1}}]\right|\right\|_{L^\alpha(P_y)}\\
\le C(\alpha)\sqrt{T}
\left(\sqrt{h}+
\frac{\sup_{z\in I}\left( E_z[H_{z-a_h(z), z + a_h(z)}(Y)]\right)}{\sqrt{h}}
\right).
\end{multline}
%\blue{
%In particular, if $\sup_{y\in I } \int_{(y-a_h(y),y+a_h(y))} (a_h(y)-|u-y|)\,m(du)
%\in o(\sqrt{h})$ as $h\to 0$,
%then it holds that
%$$
%\sup_{k\in \{1,\ldots,\lfloor T/h\rfloor \}}\left|\tau^h_k-\sum_{n=1}^k E[\rho^h_n|\mathcal F_{\tau^h_{n-1}}]\right|
%\xrightarrow[]{L^\alpha(P_y)}0,
%\quad h\to 0.
%$$
%}
\end{propo}

\begin{proof}
Without loss of generality we consider $\alpha\in[2,\infty)$.
Throughout the proof we fix $T\in(0,\infty)$, $y\in I^\circ$,
$h\in (0,\ol h)$ and let $N=\lfloor T/h\rfloor$.
For all $k\in \{0,\ldots, N\}$ we define $\mathcal G_k=\mathcal F_{\tau^h_k}$ and $M_k=\tau^h_k-\sum_{n=1}^k E[\rho^h_n|\mathcal G_{n-1}]$. Notice that $(M_k)_{k\in \{0,\ldots, N\}}$ is a $(\mathcal G_k)_{k\in \{0,\ldots, N\}}$-martingale. The Burkholder-Davis-Gundy inequality ensures that there exists a constant $\wt C(\alpha) \in (0,\infty)$ (only depending on $\alpha$) such that
\begin{equation*}
\begin{split}
E_y\left[\sup_{k\in \{1,\ldots, N\}}\left|\tau^h_k-\sum_{n=1}^k E[\rho^h_n|\mathcal G_{n-1}]\right|^\alpha\right]&\leq \wt C(\alpha) E_y\left[\left(\sum_{k=1}^N (M_k-M_{k-1})^2\right)^{\frac{\alpha}{2}} \right]\\
&=\wt C(\alpha) E_y\left[\left(\sum_{k=1}^N  \left(\rho^h_k-E[\rho^h_k|\mathcal G_{k-1}]\right)^2\right)^{\frac{\alpha}{2}} \right].
\end{split}
\end{equation*}
This, together with Jensen's inequality, proves that
\begin{equation}\label{eq:01091017}
\begin{split}
E_y\left[\sup_{k\in \{1,\ldots, N \}}\left|\tau^h_k-\sum_{n=1}^k E[\rho^h_k|\mathcal G_{k-1}]\right|^\alpha\right]
&\leq \wt C(\alpha) N^{\frac{\alpha}{2}-1} \sum_{k=1}^N E_y\left[ \left|\rho^h_k-E[\rho^h_k|\mathcal G_{k-1}]\right|^{\alpha} \right]\\
&\le 
\wt C(\alpha) 2^\alpha N^{\frac{\alpha}{2}-1} \sum_{k=1}^N  E_y\left[ \left|\rho^h_k\right|^\alpha\right]\\
&\le 
\wt C(\alpha) 2^\alpha N^{\frac{\alpha}{2}} \sup_{k\in \N}  E_y\left[ \left|\rho^h_k\right|^\alpha\right]\\
&\le 
\wt C(\alpha) 2^\alpha T^{\frac{\alpha}{2}} \left(\frac{\sup_{k \in \N} \left\| \rho^h_k\right\|_{L^\alpha(P_y)}}{\sqrt{h}} \right)^\alpha.
\end{split}
\end{equation}
Then Lemma~\ref{l:st_ui} proves~\eqref{eq:lalphamart}.
\end{proof}

\begin{propo}\label{theo:conv_A}
Let $\alpha \in (0,\infty)$.
Then, for all $T\in(0,\infty)$, $y\in I^\circ$ and $h\in (0,\ol h)$, it holds that
\begin{equation}\label{eq:lalphapred}
\begin{split}
\left\|\sup_{k\in \{1,\ldots, \lfloor T/h \rfloor \}}\left|\left(\sum_{n=1}^k E[\rho^h_n|\mathcal F_{\tau^h_{n-1}}]\right)-kh\right|\right\|_{L^\alpha(P_y)}
&\le \frac{T}{h} \sup_{z\in I_h}
\left|
E_z[H_{z-a_h(z),z+a_h(z)}(Y)]
-h
\right|.
\end{split}
\end{equation}
%\blue{
%In particular, if Condition~(A) is satisfied, then we have 
%$$
%\sup_{k\in \{1,\ldots, \lfloor T/h \rfloor \}}\left|\left(\sum_{n=1}^k E[\rho^h_n|\mathcal F_{\tau^h_{n-1}}]\right)-kh\right|
%\xrightarrow[]{L^\alpha(P_y)}0,
%\quad h\to 0.
%$$
%}
\end{propo}

\begin{proof}
Without loss of generality we consider $\alpha\in[1,\infty)$.
Throughout the proof we fix $T\in(0,\infty)$, $y\in I^\circ$,
$h\in (0,\ol h)$ and let $N=\lfloor T/h\rfloor$.
For all $k\in \{0,\ldots, N\}$ we define $\mathcal G_k=\mathcal F_{\tau^h_k}$.
The triangle inequality ensures that
\begin{equation}\label{eq:231119a}
\begin{split}
\left\|\sup_{k\in \{1,\ldots, N \}}\left|\left(\sum_{n=1}^k E[\rho^h_n|\mathcal G_{n-1}]\right)-kh\right|\right\|_{L^\alpha(P_y)}
&=
\left\|\sup_{k\in \{1,\ldots,N \}}\left|\left(\sum_{n=1}^k E[\rho^h_n|\mathcal G_{n-1}]-h\right)\right|\right\|_{L^\alpha(P_y)}
\\
&\le \sum_{n=1}^N
\left\|  (E[\rho^h_n|\mathcal G_{n-1}]-h)\right\|_{L^\alpha(P_y)}.
\end{split}
\end{equation}
By Proposition~\ref{main thm sec 2}, on the event $\{Y_{\tau^h_{n-1}} \in I_h\}$ we have 
\begin{equation}
\begin{split}
&\left|E_y[\rho^h_n|\mathcal G_{n-1}]-h\right| \\
&=
\left|
\frac{1}{2}\int_{(Y_{\tau^h_{n-1}}-a_h(Y_{\tau^h_{n-1}}),Y_{\tau^h_{n-1}}+a_h(Y_{\tau^h_{n-1}}))} (a_h(Y_{\tau^h_{n-1}})-|u-Y_{\tau^h_{n-1}}|)\,m(du)
-h
\right|
\\
&\le 
\sup_{z\in I_h } 
\left|
\frac{1}{2}\int_{(z-a_h(z),z+a_h(z))} (a_h(z)-|u-z|)\,m(du)
-h
\right|\\
&=\sup_{z\in I_h} 
\left|
E_z[H_{z-a_h(z),z+a_h(z)}(Y)]
-h
\right|.
\end{split}
\end{equation}
On the event $\{Y_{\tau^h_{n-1}} \notin I_h\}$ we have $|E_y[\rho^h_n|\mathcal G_{n-1}]-h| = 0$. Therefore, 
%This implies
\begin{equation}
\begin{split}
\left\|\sup_{k\in \{1,\ldots, N \}}\left|\left(\sum_{n=1}^k E[\rho^h_n|\mathcal G_{n-1}]\right)-kh\right|\right\|_{L^\alpha(P_y)}
&\le N \sup_{z\in I_h } 
\left|
E_z[H_{z-a_h(z),z+a_h(z)}(Y)]
-h
\right|
\\
&\le \frac{T}{h} \sup_{z\in I_h } 
\left|
E_z[H_{z-a_h(z),z+a_h(z)}(Y)]
-h
\right| .
\end{split}
\end{equation}
This proves~\eqref{eq:lalphapred}.
\end{proof}

By combining the two preceding theorems we obtain a result
about uniform in $k$ convergence of the embedding stopping times
$(\tau^h_k)$ in spaces $L^\alpha(P_y)$, as $h\to0$.
To this end, we impose a slightly stronger condition than
Condition~(A), namely,
\begin{equation}\label{eq:26112018a2}
\sup_{y\in I_h}  \left|
\frac{1}{2}\int_{(y-a_h(y),y+a_h(y))} (a_h(y)-|u-y|)\,m(du)
-h
\right|
\in o(h), \quad h \to 0.
\end{equation}

\begin{corollary}\label{cor:26112018a1}
Assume~\eqref{eq:26112018a2}.
Let $\alpha \in (0,\infty)$, $T\in (0,\infty)$ and $y\in I^\circ$.
Then it holds
$$
\sup_{k\in \{1,\ldots, \lfloor T/h \rfloor \}}\left|\tau^h_k-kh\right|
\xrightarrow[]{L^\alpha(P_y)}0,
\quad h\to 0.
$$
\end{corollary}

\begin{proof}
The proof is an application of Propositions~\ref{coro sec 3} and~\ref{theo:conv_A}.
The fact that the right-hand side of~\eqref{eq:lalphapred}
converges to zero as $h\to0$
is a direct consequence of~\eqref{eq:26112018a2}
(also recall Remark~\ref{rem:03012020a1}).
Similarly, \eqref{eq:26112018a2} implies
$$
\frac{\sup_{z\in I_h}\left( E_z[H_{z-a_h(z), z + a_h(z)}(Y)]\right)}{\sqrt{h}}\to0,\quad h\to0.
$$
The remaining property
$$
\frac{\sup_{z\in I\setminus I_h}\left( E_z[H_{z-a_h(z), z + a_h(z)}(Y)]\right)}{\sqrt{h}}\to0,\quad h\to0
$$
(cf.~\eqref{eq:lalphamart}), follows from the fact that,
by the definition of $I_h$, we have
$$
 E_z[H_{z-a_h(z), z + a_h(z)}(Y)]\le h,\quad z\in I\setminus I_h.
$$
This concludes the proof.
\end{proof}

\begin{proof}[Proof of Theorem~\ref{main thm sec 1}]
For any $h\in(0,\ol h)$, we define the continuous-time process
$Y^h=(Y^h_t)_{t\in[0,\infty)}$
by linear interpolation of $(Y_{\tau^h_k})_{k\in\bbN_0}$.
More precisely, we set
$$
Y^h_t=Y_{\tau^h_{\lfloor t/h\rfloor}}
+\left(t/h-\lfloor t/h\rfloor\right)
\left(Y_{\tau^h_{\lfloor t/h\rfloor+1}}-Y_{\tau^h_{\lfloor t/h\rfloor}}\right),
\quad t\in[0,\infty).
$$
Notice that $Y^h_{kh}=Y_{\tau^h_k}$, $k\in\bbN_0$,
and Proposition~\ref{main thm sec 2}
easily extends to
$$
\Law_{P_y}
\left(Y^h_t;t\in[0,\infty)\right)
=\Law_P
\left(X^{h,y}_t;t\in[0,\infty)\right).
$$
Therefore, in order to prove Theorem~\ref{main thm sec 1},
it is sufficient to show that the processes
$Y^h=(Y^h_t)_{t\in[0,\infty)}$
converge to the process
$Y=(Y_t)_{t\in[0,\infty)}$
in probability $P_y$ uniformly on compact intervals, i.e.,
that, for all $T\in(0,\infty)$, it holds
$$
\|Y^h-Y\|_{C[0,T]}\xrightarrow[]{P_y}0,\quad h\to0,
$$
where $\|\cdot\|_{C[0,T]}$ denotes the sup norm in $C([0,T],\bbR)$.
In what follows, we use the notation
\begin{equation}\label{eq:27112018a1}
Y^h\xrightarrow[]{\ucp(P_y)}Y,\quad h\to0,
\end{equation}
for this mode of convergence.

\smallskip
\textbf{1.}
In the first step, we prove~\eqref{eq:27112018a1}
under assumption~\eqref{eq:26112018a2},
which is stronger than Condition~(A).
To this end, fix $T\in(0,\infty)$ and $\eps>0$.
Take an arbitrary $T'\in(0,\infty)$, $T'>T$,
and choose $\delta\in\left(0,\frac{T'-T}2\right)$
such that $P_y(A(\delta))>1-\frac\eps2$, where
$$
A(\delta) = \{ |Y_t-Y_s| < \frac{\varepsilon}{2} \text{ for all } s,t \in [0,T'] \text{ such that } |t-s| < 3 \delta \}. 
$$
Corollary~\ref{cor:26112018a1} implies
$$
\sup_{k\in \{1,\ldots, \lfloor T'/h \rfloor \}}\left|\tau^h_k-kh\right|
\xrightarrow[]{P_y}0,
\quad h\to 0,
$$
hence, there exists $\gamma\in(0,\ol h)$ such that
$P_y(C(h,\delta))>1-\frac\eps2$
whenever $h\in(0,\gamma)$, where
$$
C(h, \delta) = \left\{ \sup_{k\in \{1,\ldots, \lfloor T'/h \rfloor \}}\left|\tau^h_k-kh\right| < \delta \right\}.
$$
A somewhat tedious check shows that, if $h\in(0,\delta)$,
$$
\|Y^h-Y\|_{C[0,T]}<\eps\quad\text{on }A(\delta)\cap C(h,\delta).
$$
Thus, we get
$P_y(\|Y^h-Y\|_{C[0,T]}>\eps)<\eps$
whenever $h\in(0,\gamma\wedge\delta)$.
This completes the proof of the first step.

\smallskip
\textbf{2.}
We now prove~\eqref{eq:27112018a1} under Condition~(A).
Consider strictly monotone sequences $\{l_n\}_{n\in\bbN}$
and $\{r_n\}_{n\in\bbN}$ with $l_n\searrow l$ and $r_n\nearrow r$.
We define compact subintervals $K_n$ of $I^\circ$
by setting $K_n=[l_n,r_n]$, $n\in\bbN$,
and modified scale factors $\wt a^n_h\colon\ol I\to[0,\infty)$
by setting
$$
\wt a^n_h(y)=\begin{cases}
a_h(y),&y\in K_n,\\
\wh a_h(y),&y\in\ol I\setminus K_n,
\end{cases}
$$
$n\in\bbN$, $h\in(0,\ol h)$,
where the scale factors $\wh a_h$, $h\in(0,\ol h)$,
are the ones from the EMCEL algorithm
(recall~\eqref{sf absorbing case}).
Let $(\wt\tau^{n,h}_k)_{k\in\bbN_0}$
be the associated sequences of the embedding stopping times
and $\wt Y^{n,h}=(\wt Y^{n,h}_t)_{t\in[0,\infty)}$
the analogues of the process
$Y^h=(Y^h_t)_{t\in[0,\infty)}$
for the modified scale factors $\wt a^n_h$,
$n\in\bbN$, $h\in(0,\ol h)$.
Since the scale factors $(a_h)_{h\in(0,\ol h)}$
satisfy Condition~(A), the modified scale factors
$(\wt a^n_h)_{h\in(0,\ol h)}$
satisfy~\eqref{eq:26112018a2}
for each $n\in\bbN$.
By the first step of the proof,
\begin{equation}\label{eq:28112018a1}
\wt Y^{n,h}\xrightarrow[]{\ucp(P_y)}Y,\quad h\to0,
\end{equation}
for any fixed $n\in\bbN$.

Fix $T\in(0,\infty)$ and $\eps>0$.
For any $n\in\bbN$, we define the events
\begin{align*}
A_n&=\{H_{l,r}(Y)\le T+2\text{ and }\exists\,t\in[H_{l_n,r_n}(Y),H_{l,r}(Y)]
\text{ such that }|Y_t-Y_{H_{l,r}(Y)}|>\eps\},\\
B_n&=\{H_{l,r}(Y)>T+2\text{ and }H_{l_n,r_n}(Y)\le T+1\}.
\end{align*}
Notice that the expression $Y_{H_{l,r}(Y)}$
in the above formula for $A_n$
is well-defined and finite.
Indeed, this is the position of $Y$
at an accessible boundary
(because $H_{l,r}(Y)\le T+2<\infty$),
while an infinite boundary cannot be accessible
(because $Y$ is in natural scale).
As $H_{l_n,r_n}(Y)\nearrow H_{l,r}(Y)$ $P_y$-a.s.,
as $n\to\infty$, and $Y$ is continuous, we can choose
a sufficiently big $n_0\in\bbN$ such that
$P_y(A_{n_0})<\frac\eps3$ and $P_y(B_{n_0})<\frac\eps3$.
We also take an arbitrary $T'\in(T,T+1)$.
Corollary~\ref{cor:26112018a1}
applied to the modified scale factors
$(\wt a^{n_0}_h)_{h\in(0,\ol h)}$,
which satisfy~\eqref{eq:26112018a2},
yields that there exists $\gamma>0$ such that,
for any $h\in(0,\gamma)$, we have
$$
P_y\left(\wt\tau^{n_0,h}_{\lfloor T'/h\rfloor}\le T+1\right)>1-\frac\eps3.
$$
For $h\in(0,\ol h)$, we define the event
$$
C_h=\left\{\wt\tau^{n_0,h}_{\lfloor T'/h\rfloor}\le T+1\right\}
\cap\left(A_{n_0}\cup B_{n_0}\right)^c
$$
(the notation $D^c$ means the complement of an event~$D$).
Notice that $P_y(C_h)>1-\eps$ whenever $h\in(0,\gamma)$.
Furthermore, on $C_h$ we have either
$$
H_{l,r}(Y)>T+2,\;
H_{l_{n_0},r_{n_0}}(Y)>T+1,\;
\text{hence }
Y^h_t=\wt Y^{n_0,h}_t,\;
t\in[0,T],
$$
or
\begin{align*}
&H_{l,r}(Y)\le T+2,\;
|Y_t-Y_{H_{l,r}(Y)}|\le\eps\;
\text{for }t\in[H_{l_{n_0},r_{n_0}}(Y),H_{l,r}(Y)],\\
&\text{hence }
\left|Y^h_t-Y_t\right|\le\eps
\text{ whenever }
Y^h_t\ne\wt Y^{n_0,h}_t,\;
t\in[0,T].
\end{align*}
Together with~\eqref{eq:28112018a1}, this proves
$\|Y^h-Y\|_{C[0,T]}\xrightarrow[]{P_y}0$ as $h\to0$.
As $T\in(0,\infty)$ is arbitrary, we obtain~\eqref{eq:27112018a1}.
This concludes the proof.
\end{proof}

\section{Reflecting boundaries}
\label{sec:reflecting}
Throughout the preceding sections we assume that if a boundary point is accessible, then it is absorbing. In this section we explain how one can drop this assumption, i.e.\ how one can extend our functional limit theorem, Theorem~\ref{main thm sec 1}, to Markov processes with reflecting boundaries. 

The idea is to reduce the reflecting case to the inaccessible or absorbing case. Indeed, for every Markov process $Z$ with reflecting boundaries one can find a Markov process $Y$ on an extended state space and a Lipschitz function $f$ such that 
$Y$ has inaccessible or absorbing boundaries and $Z \stackrel{d}{=} f(Y)$. 

We illustrate the reduction for a Markov process $Z$ in natural scale with state space $I_Z = [l, \infty)$, where $l > -\infty$ is a reflecting boundary. We denote by $m_Z$ the speed measure of $Z$. Since $l$ is non-absorbing, it must hold that $m_Z(\{l\}) < \infty$.
Notice that $m_Z(\{l\})=0$ corresponds to instantaneous reflection,
while $m_Z(\{l\})\in(0,\infty)$
%corresponds
to slow reflection.

To proceed with the construction, we first remark that it holds
\begin{equation}\label{eq:06072018a2}
m_Z((l,l+1))<\infty.
\end{equation}
Indeed, in terms of the Feller boundary classification
(see Table~15.6.2 in \cite{KarlinTaylor:81}),
as the accessible boundary point $l$ is reflecting,
it can only be \emph{regular},
which implies~\eqref{eq:06072018a2}.

Now let $Y$ be a Markov process in natural scale with state space $I_Y = \R$ and speed measure $m_Y$ satisfying
\begin{align*}
m_Y(A) & = m_Z(A),
\quad\text{for all }A\in\cB(\bbR),\;A\subseteq(l,\infty), \\
m_Y(A) & = m_Z(2l-A),
\quad\text{for all }A\in\cB(\bbR),\;A\subseteq(-\infty,l), \\
m_Y(\{l\}) & = 2 m_Z(\{l\}),
\end{align*}
which is a valid speed measure on $I_Y=\bbR$
(i.e., \eqref{eq:06072018a1} holds)
due to~\eqref{eq:06072018a2}.
Then $l+|Y-l|$ has the same distribution as~$Z$
(see Proposition~VII.3.10 in \cite{RY}).

Let $(a_h)_{h \in (0,\ol h)}$ satisfy the assumptions of Theorem \ref{main thm sec 1}. Then $(X^h)_{h\in (0,\ol h)}$ converges in distribution to $Y$ as $h\to 0$. This implies 
that the processes $(l+|X^h-l|)_{h\in (0,\ol h)}$ converge in distribution to~$Z$ as $h\to 0$.

\bigskip
In a similar way, a Markov process $Z$ on a bounded interval $I_Z$
with endpoints $l$ and $r$ ($l<r$), where $l\in I_Z$ is reflecting
and $r$ is inaccessible (resp., absorbing),
can be reduced to a Markov process $Y$ with state space
$I_Y$, which is the interval with endpoints $2l-r$ and $r$,
where both these endpoints are inaccessible (resp., absorbing).

\bigskip
A Markov process $Z$ with two reflecting boundaries can be reduced to a Markov process with state space~$\R$. To explain this, suppose for simplicity that the state space of $Z$ is $[0,1]$. Define $Y$ as the Markov process on $\R$ with speed measure $m_Y$ satisfying
\begin{align*}
m_Y(A) = m_Y(-A) = m_Z(A), &\text{ for all } A\in \cB\big( (0,1) \big), \\
m_Y(A + 2k) = m_Z(A), &\text{ for all } A\in \cB\big((-1,0)\cup(0,1)\big) \text{ and } k \in \IZ, \\
m_Y(\{2k\}) = 2 m_Z(\{0\}), &\text{ for all } k \in \IZ, \\
m_Y(\{2k+1\}) = 2 m_Z(\{1\}), &\text{ for all } k \in \IZ.
\end{align*}   
Let $f\colon\R\to[0,1]$ be the periodic function
with period $2$ satisfying $f(x)=|x|$, $x\in[-1,1]$.
Then the process $f(Y)$ has the same distribution as~$Z$
(cf.\ Proposition~VII.3.10 in \cite{RY}).

\section{Examples with sticky points}\label{sec:examples}
In this section we apply Theorem~\ref{main thm sec 1}
to sticky Brownian motions on $\bbR$
and on $[0,\infty)$, where the sticky point is zero.
In the latter case one also speaks about slow (or sticky) reflection at $0$.
Recent years have witnessed an increased interest in the sticky Brownian motion and related processes, see
\cite{KSS2011},
\cite{Bass2014},
\cite{ep2014},
\cite{HajriCaglarArnaudon:17},
\cite{CanCaglar2019}
and references therein.
Newly, diffusions with slow reflection
were applied in \cite{EberleZimmer2019}
to provide bounds (via sticky couplings)
for the distance between two
multidimensional diffusions with different drifts.
Stickiness is a convenient concept for modeling repulsive interactions between particles, and, motivated by natural questions from physics, it is discussed in multi- and infinitedimensional situations in
\cite{FGV2016},
\cite{GV2017},
\cite{GV2018},
\cite{Konarovskyi2017},
\cite{KvR2017}.
Diffusions with slow reflection also attracted interest in
economic theory, where such processes characterize
optimal continuation values
in dynamic principal-agent problems (see, e.g., \cite{zhu2012optimal} and \cite{piskorski2016optimal}).

On the contrary,
the literature on \emph{approximations} of diffusions with atoms in the speed measure is scarce. We remark that \cite{Amir1991} provides
a sequence of random walks that converges in distribution to the Brownian motion on $\bbR$ sticky at zero. The random walks considered there are forced to stay in zero for some time whenever they visit zero. In contrast to our approach, the approximating processes are not Markov chains.
\cite{fushiya2010weak} constructs
Markov chains that converge in distribution to 
the Brownian motion on $[0,\infty)$ with slow reflection 
at~$0$. The approximating Markov chains considered there exhibit sticky behavior in 
zero in the sense that once the Markov chains reach zero they stay there with positive probability 
also in the next time period. The recent work \cite{bou2019sticky} proposes an approximation of the Brownian motion on $[0,\infty)$ with slow reflection 
at~$0$ by continuous-time pure jump Markov processes $Y^\delta$, $\delta\in (0,\infty)$, with uniform grids $\{0,\delta,2\delta,\ldots\}$ as state spaces. The jump times are exponentially distributed and the mean waiting time at the interior points $\{\delta,2\delta,\ldots\}$ is of order $\delta^2$ whereas it is of order $\delta$ at the origin~$0$.

\subsection{Brownian motion on $\R$ with sticky point~$0$}
\label{subsec:sticky_at_0}
%% dictionary: $\beta=\frac{1}{\sigma^2}$ and $\kappa=\frac{1}{2\theta}$.

Brownian motion on $\bbR$ sticky at $0$ is a Markov process $Y$ in natural scale with state space $I=\bbR$ and speed measure
\begin{equation}\label{eq:sm_sticky_bm}
m(dx)=\frac{2}{\sigma^2}\,\lambda(dx)+ \frac{2}{\theta}\,\delta_0(dx),
\end{equation}
where $\sigma, \theta \in (0,\infty)$ and $\lambda(dx)$ denotes the Lebesgue measure.\footnote{Since there exist different conventions concerning the normalization of the speed measure (cf.\ Footnote~\ref{fn:conv_sm}), our representation of $m$ in \eqref{eq:sm_sticky_bm} may differ by a factor of $2$ from related representations found in the literature (cf., e.g., \cite{Bass2014}).}
Such a process $Y$ behaves like $\sigma$ times a Brownian motion outside zero, but spends a positive amount of time at zero having no intervals of zeros.
Notice that the bigger $\theta$ is, the less time $Y$ spends at zero;
$\theta=\infty$ corresponds to a standard Brownian motion (times~$\sigma$).

It is instructive to compute the function $q(y,x)$, $y,x\in\bbR$,
of~\eqref{eq:def_q}
$$
q(y,x)=\begin{cases}
\frac{(x-y)^2}{\sigma^2}+2\frac{x^-}{\theta}&\text{if }y>0,\\[1mm]
\frac{x^2}{\sigma^2}+\frac{|x|}{\theta}&\text{if }y=0,\\[1mm]
\frac{(x-y)^2}{\sigma^2}+2\frac{x^+}{\theta}&\text{if }y<0
\end{cases}
$$
($x^+=\max\{x,0\}$, $x^-=-\min\{x,0\}$)
and to observe that, for any $y\in\bbR$,
the function $q(y,\cdot)$ has a kink at zero.

We now determine, for every $h \in (0,\infty)$, a function $\wh a_h\colon \R \to (0,\infty)$ such that the associated Markov chain $(\wh X^h_{hk})_{k \in \mathbb{N}_0}$, defined in \eqref{eq:def_X}, belongs to EMCEL$(h)$. Indeed, one can explicitly
determine for all $y\in \R$ the real number $\wh a_h(y)$ satisfying
\begin{equation}\label{eq:equality_for_sf}
\frac{1}{2}\int_{(y-\wh a_h(y),y+\wh a_h(y))}(\wh a_h(y)-|u-y|)\, m(du)=h.
\end{equation}
We state the closed-form representations of $\wh a_h$ in the next Lemma.

\begin{lemma}
For all $h\in (0,\infty)$ and $y\in \R$ let
$$
\wh a_h(y)=\begin{cases}
\sigma \sqrt{h}
&\text{if } |y|\ge \sigma \sqrt{h},\\[1mm]
\sigma\left(\sqrt{h+\frac{|y|}{\theta}+\left(\frac{\sigma}{2\theta}\right)^2}-\frac{\sigma}{2\theta}\right)&\text{if } |y|< \sigma \sqrt{h}.
\end{cases}
$$
Then Equation \eqref{eq:equality_for_sf} is satisfied for all $h\in (0,\infty)$ and $y\in \R$.
\end{lemma}

\begin{proof}
Throughout the proof fix $h\in (0,\infty)$ and $y\in \R$. For every $a\in [0,\infty)$ it holds that
$$
\frac{1}{2}\int_{(y-a,y+a)}(a-|u-y|)\, m(du)
=\frac{a^2}{\sigma^2}+\frac{a-|y|}{\theta}1_{(y-a,y+a)}(0)
=\frac{a^2}{\sigma^2}+\frac{a-|y|}{\theta}1_{[0,a)}(|y|).
$$
Assume first that $|y|\ge \sigma\sqrt{h}$. Then it holds that $|y|\ge \wh a_h(y)$
and hence \eqref{eq:equality_for_sf} is satisfied. Next assume that $|y| < \sigma \sqrt{h}$. In this case it holds that $\wh a_h(y)> |y|$. Moreover it holds that
\begin{equation}
\frac{\wh a_h(y)^2}{\sigma^2} +\frac{\wh a_h(y)-|y|}{\theta} =h.
\end{equation}
This proves \eqref{eq:equality_for_sf} in the case $|y| < \sigma\sqrt{h}$. The proof is thus completed.
\end{proof}
Since $\wh a_h$ satisfies Equation \eqref{eq:equality_for_sf} exactly, it follows that Condition~(A) is satisfied. 
Theorem~\ref{main thm sec 1} implies that
the processes $(X^h)$ %(see definition \eqref{eq:13112017a1}) 
converge in distribution to $Y$ as $h\to 0$.
Figure~\ref{fig:sbm1} depicts two realizations
of a Brownian motion on $\bbR$ sticky at $0$
with $\sigma=1$ and different values for $\theta$
as well as the empirical distribution function of~$\wh X^h_1$ with $h=10^{-3}$.

\begin{figure}[htb!]
\centering
\includegraphics[width=0.49\textwidth]{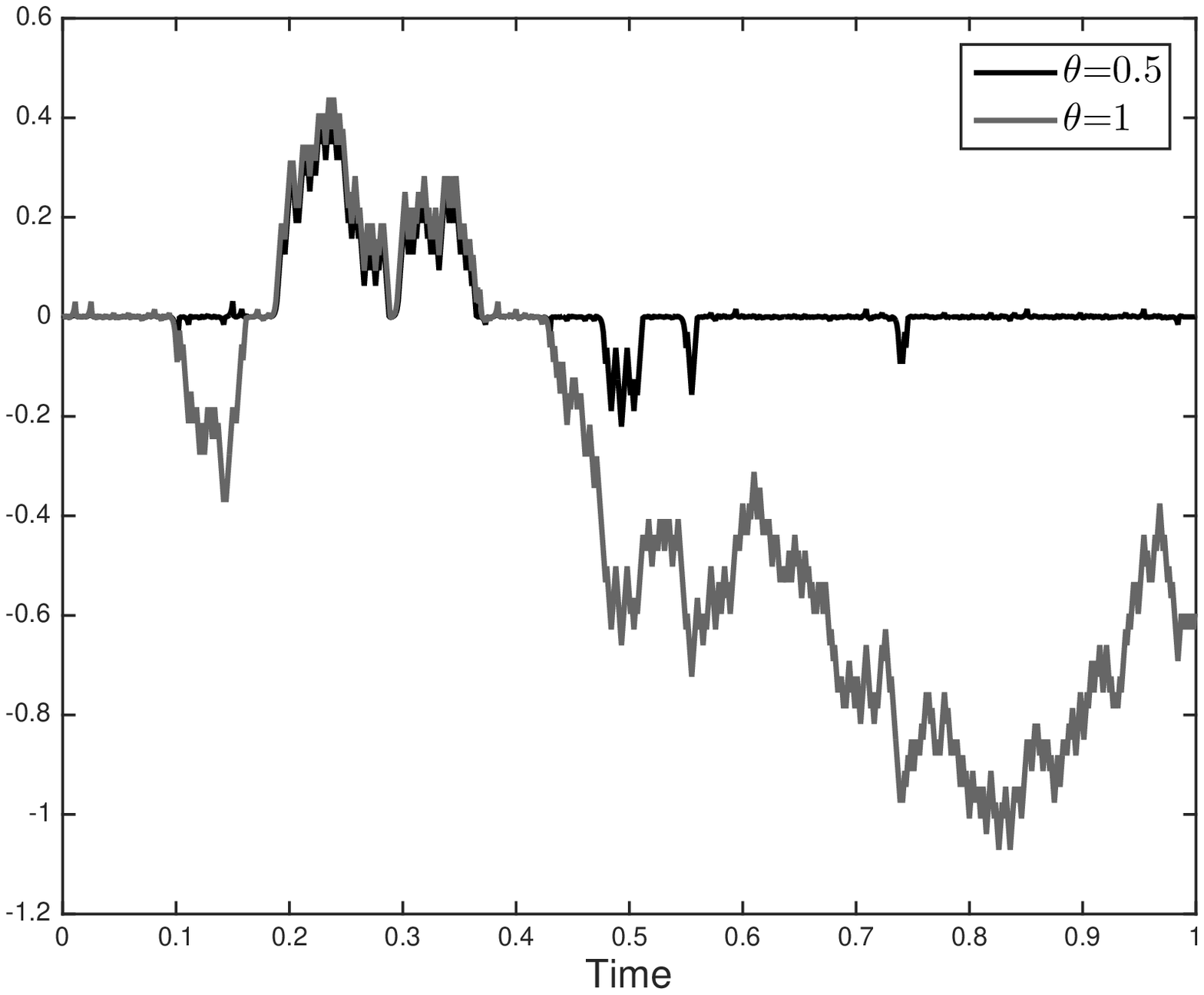}
\includegraphics[width=0.49\textwidth]{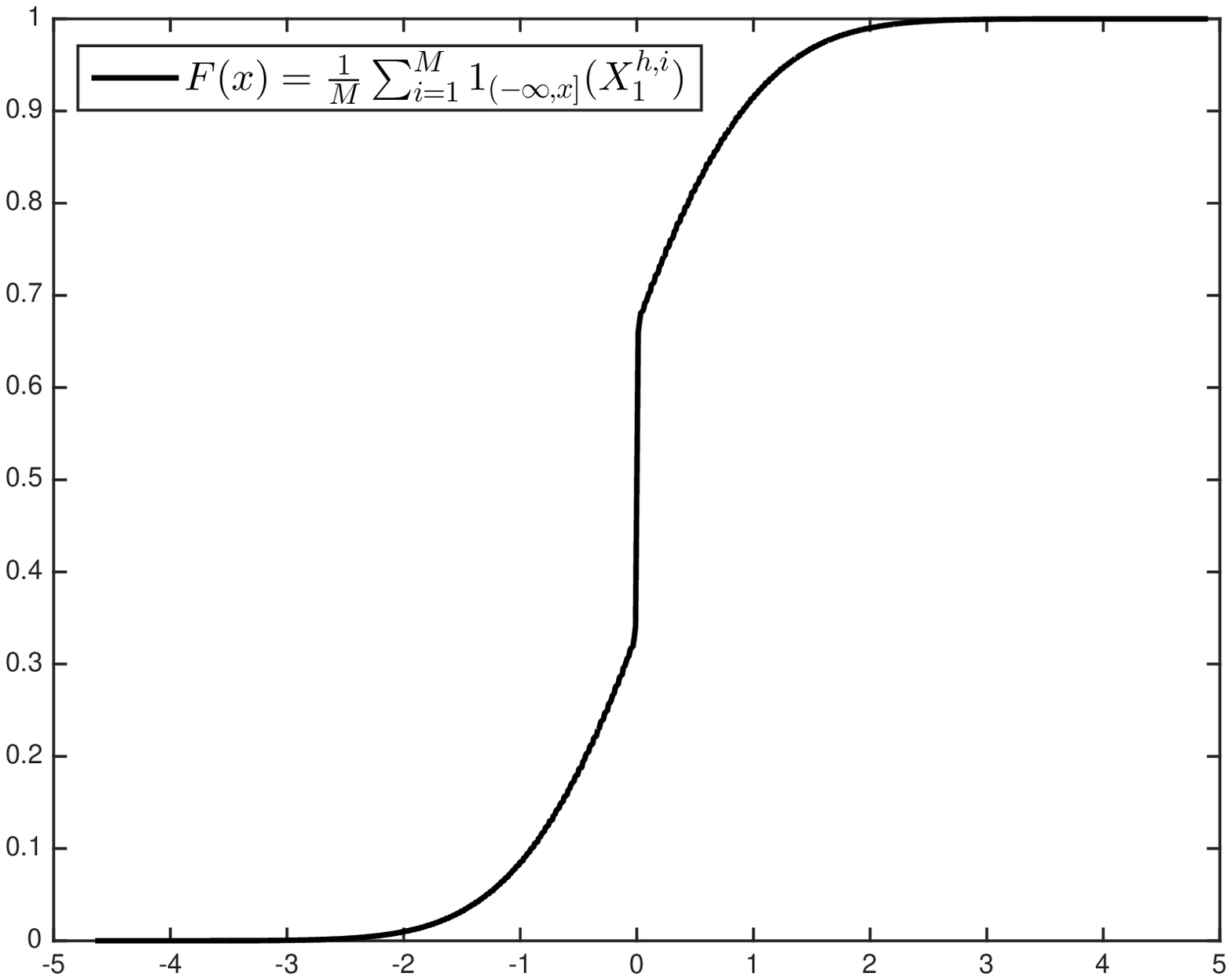}
{\caption{\footnotesize Left: Two trajectories of the approximation of the sticky Brownian motion. The black line depicts one realization of $(\wh X^h_{t})_{t\in [0,1]}$ with $h=10^{-3}$, $\sigma=1$ and $\theta=0.5$. The gray line shows one realization of $(\wh X^h_{t})_{t\in[0,1]}$ with $h=10^{-3}$, $\sigma=1$ and $\theta=1$. Both trajectories are generated by the same sample of random increments $(\xi_k)_{k\in \N}$. Observe that the smaller the value of $\theta$ is, the more the process sticks to~$0$.
Right: Empirical distribution function of the approximation of the sticky Brownian motion. The figure depicts the function $F\colon \R \to [0,1]$, $F(x)=\frac{1}{M}\sum_{i=1}^M 1_{(-\infty,x]}(\wh X^{h,i}_1)$, where $(\wh X^{h,i}_1)_{i\in \{1,\ldots,M\}}$ are $M=10^6$ independent realizations of $\wh X^h_1$ with $h=10^{-3}$, $\sigma=1$ and $\theta=1$. Observe that a jump at $0$ becomes apparent. This reflects the fact that the (weak) limit $Y_1$ of $\wh X^h_1$ is with positive probability equal to~$0$. Notice that the distribution function of $Y_1$ is known in closed form (cf.\ Lemma~\ref{lem:distr_funct_sbm} below). We refrain from plotting it in the figure on the right-hand side because with the current image scaling it is visually nearly indistinguishable from~$F$.}\label{fig:sbm1}}
\end{figure}

\subsection{Brownian motion on $[0,\infty)$ with slow reflection at~$0$}\label{subsec:sticky_reflection}
In this section we consider a Brownian motion on $[0,\infty)$ with slow reflection at $0$. 
We first define this process, as in Warren \cite{Warren97}, as the solution of SDE~\eqref{eq:sde_sbm} below.
We subsequently show that its distribution is identical to the distribution of $|Y|$, where $Y$ is the general diffusion analyzed in Section~\ref{subsec:sticky_at_0}. From this perspective, the main difference between the processes studied in this section and those in Section~\ref{subsec:sticky_at_0} is the state space.

Let $\sigma,\theta\in (0,\infty)$.
According to Theorem~IV.7.2 in \cite{ikeda2014stochastic}
the stochastic differential equation
\begin{equation}\label{eq:sde_sbm}
dZ_t=\theta 1_{\{Z_t=0\}}\,dt +\sigma 1_{\{Z_t>0\}}\,dW_t, \quad Z_0=0,
\end{equation}
possesses a weak solution that is unique in law.
However, it is worth noting that neither existence of a strong
solution nor pathwise uniqueness hold for~\eqref{eq:sde_sbm}
(see \cite{ep2014} and references therein).
The next result shows that $Z$ is a regular diffusion on $[0,\infty)$ and identifies the associated speed measure.

\begin{lemma}\label{lem:sbm_speed_measure}
The solution $Z$ of~\eqref{eq:sde_sbm} is a regular continuous strong Markov process in natural scale with state space $I_Z=[0,\infty)$  and with speed measure
\begin{equation}\label{eq:12072018a2}
m_Z(dz)
=\frac2{\sigma^2}\,\lambda(dz)
+\frac{1}{\theta}\,\delta_0(dz).
\end{equation}
\end{lemma}

\begin{proof}
Strong Markov property of $Z$ is implied by the uniqueness in law for~\eqref{eq:sde_sbm}.
Clearly, $Z$ is regular with state space $I_Z=[0,\infty)$
and in natural scale.
By It\^o's formula, we have
$$
f(Z_t)=f(Z_0)+\int_0^t
\left(\theta f'(0) 1_{\{Z_s=0\}}
+\frac{\sigma^2}2 f''(Z_s) 1_{\{Z_s>0\}}\right)\,ds
+\int_0^t \sigma f'(Z_s) 1_{\{Z_s>0\}}\,dW_s
$$
for $C^2$ functions $f\colon[0,\infty)\to\bbR$.
Therefore, the generator $\cA$ of $Z$
takes the form
\begin{equation}\label{eq:12072018a1}
\cA f(z)=\begin{cases}
\theta f'(0)&\text{if }z=0,\\
\frac{\sigma^2}2 f''(z)&\text{if }z>0
\end{cases}
\end{equation}
for $f\in C^2_0([0,\infty))$
(this means that the function itself and
its first and second derivative
vanish at infinity)
satisfying the boundary condition
$$
\theta f'(0)=\frac{\sigma^2}2 f''(0).
$$
By Theorem~VII.3.12 in \cite{RY},
we have $\cA f(z)=\frac{d}{dm_Z} f'(z)$
in the interior of the state space, i.e., for $z>0$,
while, by Proposition~VII.3.13 in \cite{RY}, it holds
$f'(0)=m(\{0\})\cA f(0)$
on the boundary.
Together with~\eqref{eq:12072018a1},
this implies~\eqref{eq:12072018a2}
and concludes the proof.
\end{proof}

It follows from Section~\ref{sec:reflecting} that $Z \stackrel{d}{=} |Y|$, where $Y$ is a diffusion in natural scale with state space $I_Y=\R$ and speed measure
$m_Y(dz)=\frac2{\sigma^2}\,\lambda(dz)+\frac{2}{\theta}\,\delta_0(dz)$,
i.e., $Y$ is the process studied in
Section~\ref{subsec:sticky_at_0}
(cf.~\eqref{eq:sm_sticky_bm}).
In particular, $Z$ can be approximated by $(|\wh X^h|)_{h\in (0,\infty)}$,
where each $\wh X^h$ is the EMCEL$(h)$
constructed in Section~\ref{subsec:sticky_at_0}.

Warren \cite{Warren97} determines for all $t\in (0,\infty)$ the conditional law of $Z_t$ given the driving Brownian motion $W$. As a consequence, we obtain for all $t\in (0,\infty)$ closed form representations of the cumulative distribution function and the expected value of $Z_t$. The precise formulas are provided in Lemma~\ref{lem:distr_funct_sbm} below,
where we, without loss of generality, consider $\sigma=1$.
The notations $P_0$ for the probability measure
and $E_0$ for the corresponding expectation operator
emphasize that the formulas are given for the case $Z_0=0$.
We use these formulas to analyze the empirical rate of convergence of EMCEL approximations.
The results are presented in Figure~\ref{fig:sbm2}. 

\begin{figure}[htb!]
\centering
\includegraphics[width=0.49\textwidth]{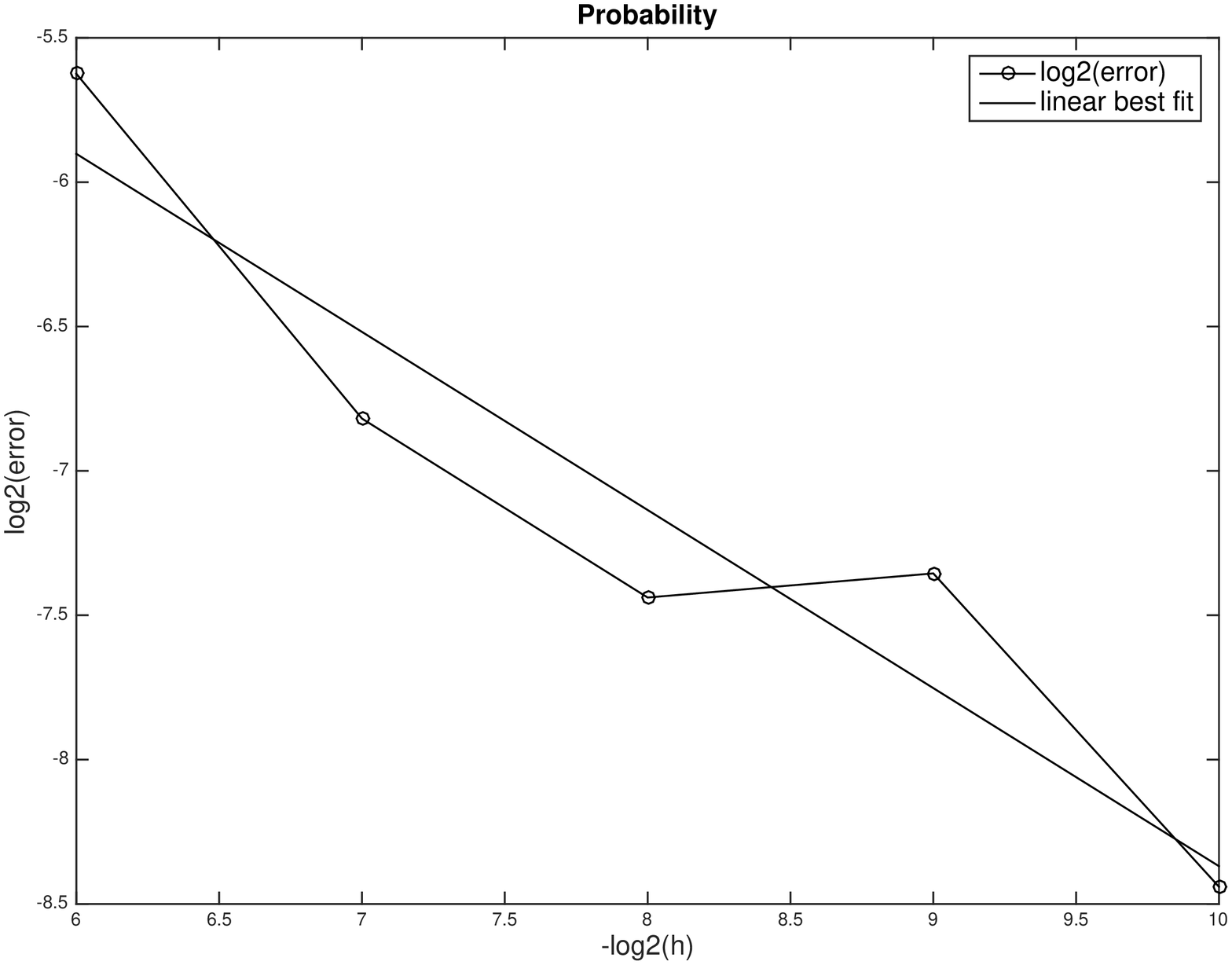}
\includegraphics[width=0.49\textwidth]{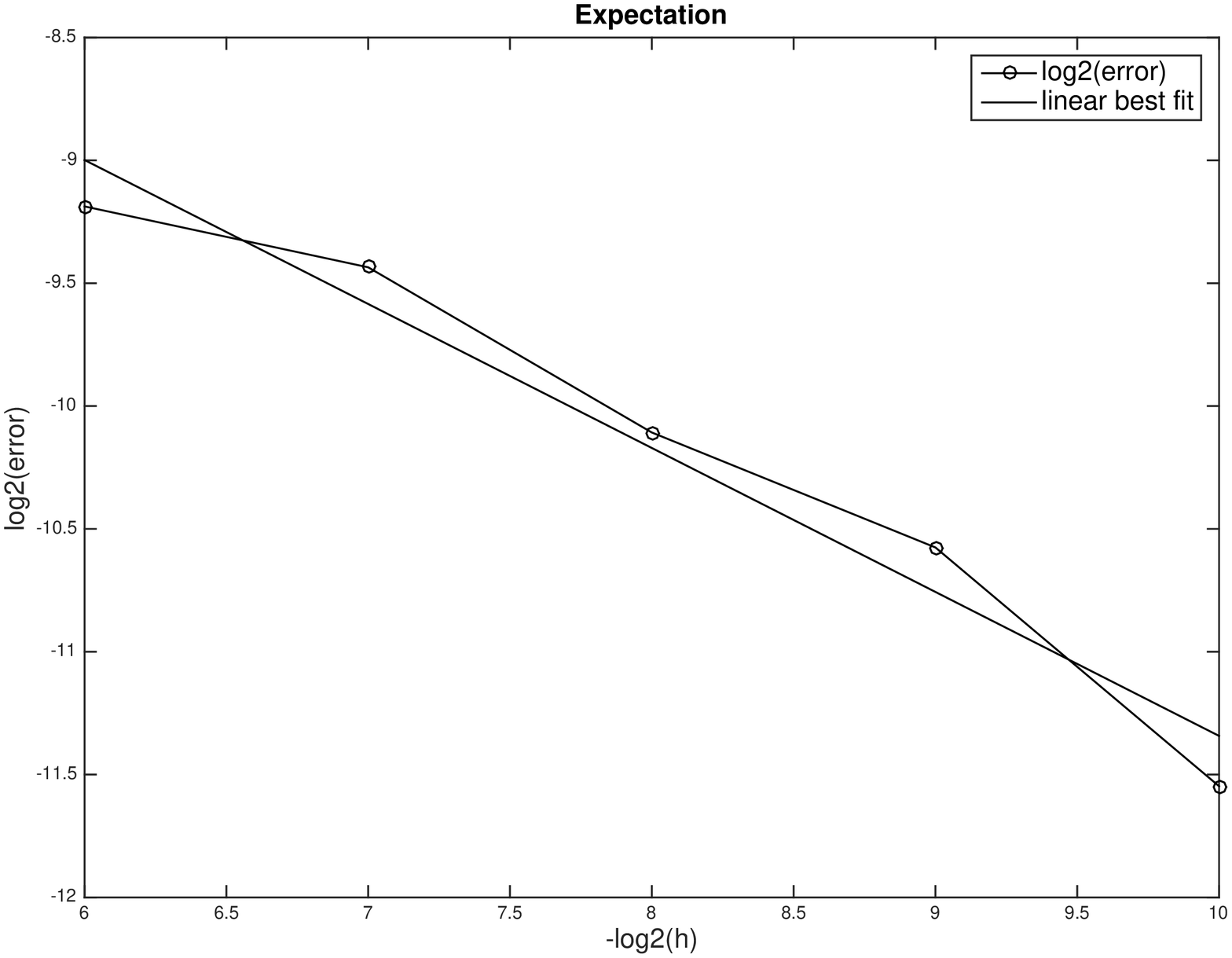}
{\caption{\footnotesize
The parameter values are $\sigma=1$ and $\theta=1/2$ for both plots.
Left: Empirical rate of convergence of the distribution function.
The circles depict the five data points 
$\left\{\left( -\log_2(h),\log_2\left|\frac{1}{M}\sum_{i=1}^M1_{[0,0.1]}(|\wh X^{h,i}_1|)-F_Z(0.1;1)\right|\right), h=2^{-6},\ldots,2^{-10}\right \}$, where $F_Z(0.1;1)=P_0[Z_1\le 0.1]\approx 0.5741$ (see Lemma~\ref{lem:distr_funct_sbm}) and $(\wh X^{h,i}_1)_{i\in \{1,\ldots,M\}}$ are $M=10^8$ independent realizations of $\wh X^h_1$.
The straight line is the linear best fit. Its slope is approximately $-0.62$.
Right: Empirical rate of convergence of the expected value. The circles depict the five data points 
$\left\{\left( -\log_2(h),\log_2\left|\frac{1}{M}\sum_{i=1}^M|\wh X^{h,i}_1|-E_0[Z_1]\right|\right), h=2^{-6},\ldots,2^{-10}\right \}$, where $E_0[Z_1]\approx 0.3210$ (see Lemma~\ref{lem:distr_funct_sbm}) and $(\wh X^{h,i}_1)_{i\in \{1,\ldots,M\}}$ are $M=10^8$ independent realizations of $X^h_1$.
The straight line is the linear best fit. Its slope is approximately $-0.59$.}\label{fig:sbm2}}
\end{figure}

\begin{lemma}\label{lem:distr_funct_sbm}
Let $Z$ be a solution of~\eqref{eq:sde_sbm} with $\sigma=1$.
For every $t\in (0,\infty)$ the cumulative distribution function $F(\cdot;t)\colon [0,\infty)\to [0,1]$ of $Z_t$ satisfies
\begin{equation}\label{eq:distr_funct_sbm}
F_Z(z;t):=P_0[Z_t\le z]=2\Phi\left(\frac{z}{\sqrt{t}} \right)-1+2e^{2\theta(z+\theta t)}\Phi\left(-2\theta \sqrt{t}-\frac{z}{\sqrt{t}}\right), \quad z\in [0,\infty),
\end{equation}
where $\Phi(x)=\int_{-\infty}^x \frac1{\sqrt{2\pi}} e^{-y^2/2}\,dy$
is the cumulative distribution function of the standard normal distribution. Moreover, it holds that
\begin{equation}\label{eq:05012020a1}
E_0[Z_t]=\sqrt{\frac{2t}{\pi}}-\frac{1}{2\theta}+\frac{e^{2\theta^2t}}{\theta}\Phi(-2\theta\sqrt{t})
=E_0[|W_t|]-\frac{1}{2\theta}+\frac{e^{2\theta^2t}}{\theta}\Phi(-2\theta\sqrt{t}).
\end{equation}
\end{lemma}

\begin{proof}
Fix $t\in (0,\infty)$ throughout the proof.
L\'evy's distributional theorem implies
$W_t+\sup_{s\in[0,t]}(-W_s)\stackrel{d}{=}|W_t|$
(see Theorem~VI.2.3 in \cite{RY}).
Then it follows from Theorem~1 in \cite{Warren97}
that for all $z\in [0,\infty)$ it holds
\begin{equation*}
\begin{split}
P_0[Z_t\le z]&=E_0\left[e^{-2\theta (W_t+\sup_{s\in[0,t]}(-W_s)-z)^+} \right]=E_0\left[1_{[0,z)}(|W_t|)+e^{-2\theta (|W_t|-z)}1_{[z,\infty)}(|W_t|) \right]\\
&=2\Phi\left(\frac{z}{\sqrt{t}} \right)-1+\frac{2}{\sqrt{2\pi}}\int_{z/\sqrt{t}}^\infty e^{-2\theta(\sqrt{t}y-z)}e^{-\frac{y^2}{2}}dy\\
&=2\Phi\left(\frac{z}{\sqrt{t}} \right)-1+\frac{2}{\sqrt{2\pi}}e^{2\theta(z+\theta t)}\int_{z/\sqrt{t}}^\infty e^{-\frac{(y+2\theta \sqrt{t})^2}{2}}dy\\
&=2\Phi\left(\frac{z}{\sqrt{t}} \right)-1+2e^{2\theta(z+\theta t)}\Phi\left(-2\theta \sqrt{t}-\frac{z}{\sqrt{t}}\right).
\end{split}
\end{equation*}
This proves \eqref{eq:distr_funct_sbm}. Moreover, this implies
that the density function of $Z_t$ starting in $0$ satisfies 
\begin{equation}
F_Z'(z;t)=\frac{2}{\sqrt{t}}\Phi'\left(\frac{z}{\sqrt{t}} \right)+4\theta e^{2\theta(z+\theta t)}\Phi\left(-2\theta \sqrt{t}-\frac{z}{\sqrt{t}}\right)
-\frac{2}{\sqrt{t}}e^{2\theta(z+\theta t)}\Phi'\left(-2\theta \sqrt{t}-\frac{z}{\sqrt{t}}\right)
\end{equation}
for all $z\in (0,\infty)$. 
Observe that for all $z\in (0,\infty)$ it holds 
\begin{equation}
\begin{split}
 e^{2\theta(z+\theta t)}\Phi'\left(-2\theta \sqrt{t}-\frac{z}{\sqrt{t}}\right)
&= \frac{1}{\sqrt{2\pi }}e^{2\theta(z+\theta t)}
e^{-\frac{\left(2\theta \sqrt{t}+\frac{z}{\sqrt{t}}\right)^2}{2}}
=  \frac{1}{\sqrt{2\pi }}
e^{-\frac{z^2}{2t}} =  \Phi'\left(\frac{z}{\sqrt{t}} \right).
\end{split}
\end{equation}
This implies for all $z\in (0,\infty)$ that
\begin{equation}\label{eq:05012020a2}
F_Z'(z;t)=4\theta e^{2\theta(z+\theta t)}\Phi\left(-2\theta \sqrt{t}-\frac{z}{\sqrt{t}}\right).
\end{equation}
This and Fubini's theorem prove that
\begin{equation}
\begin{split}
E_0[Z_t]&=\int_0^\infty z F_Z'(z;t) \, dz
=\frac{4\theta e^{2\theta^2t}}{\sqrt{2\pi}}\int_0^\infty z e^{2\theta z} \int_{-\infty}^{-2\theta \sqrt{t}-\frac{z}{\sqrt{t}}}e^{-\frac{y^2}{2}}\,dy \, dz\\
&=\frac{4\theta e^{2\theta^2t}}{\sqrt{2\pi}}\int_{-\infty}^{-2\theta \sqrt{t}} e^{-\frac{y^2}{2}} \int_0^{-y\sqrt{t}-2\theta t} ze^{2\theta z} \, dz \, dy \\
&=\frac{ e^{2\theta^2t}}{\theta \sqrt{2\pi}}\int_{-\infty}^{-2\theta \sqrt{t}} e^{-\frac{y^2}{2}} \left[1 -e^{-2\theta(y\sqrt{t}+2\theta t)}
(2\theta(y\sqrt{t}+2\theta t)+1) \right] \, dy\\
&=\frac{e^{2\theta^2 t}}{\theta}\Phi(-2\theta \sqrt{t})
-\frac{1}{\theta \sqrt{2\pi}}\int_{-\infty}^{-2\theta \sqrt{t}}e^{-\frac{(y+2\theta\sqrt{t})^2}{2}}
(2\theta\sqrt{t}(y+2\theta \sqrt{t})+1) \, dy\\
&=\frac{e^{2\theta^2 t}}{\theta}\Phi(-2\theta \sqrt{t})
-\frac{1}{\theta \sqrt{2\pi}}\int_{-\infty}^{0}e^{-\frac{y^2}{2}}
(2\theta\sqrt{t}y+1) \, dy\\
&=\frac{e^{2\theta^2 t}}{\theta}\Phi(-2\theta \sqrt{t})
-\frac{1}{2\theta}-\frac{2\sqrt{t}}{\sqrt{2\pi}}\int_{-\infty}^{0}ye^{-\frac{y^2}{2}}\, dy 
=\frac{e^{2\theta^2t}}{\theta}\Phi(-2\theta\sqrt{t})-\frac{1}{2\theta}+\sqrt{\frac{2t}{\pi}}.
\end{split}
\end{equation}
This completes the proof.
\end{proof}

\begin{remark}
Another possibility to get \eqref{eq:distr_funct_sbm}
and~\eqref{eq:05012020a1} is as follows.
An explicit formula for the transition density of a sticky Brownian motion
is provided in Part~I, Appendix~1, Section~8 of \cite{BS2002}.
This yields formula~\eqref{eq:05012020a2} for the density of~$Z_t$
(notice that the factor $4$, which is not present in \cite{BS2002},
is due to the facts that the mentioned formula in \cite{BS2002}
is given for a sticky Brownian motion on $\bbR$
and the transition density in \cite{BS2002} is given
with respect to the speed measure,
i.e., twice the Lebesgue measure outside zero).
Now \eqref{eq:05012020a1} follows by the same calculation as above,
while the distribution function~\eqref{eq:distr_funct_sbm}
can be recovered by integrating the density
and taking into account the atom at zero.
\end{remark}

\section{Brownian motion slowed down on the Cantor set}\label{sec:cantor_bm}
In this section we apply our results to construct a family of Markov chains $(X^h)_{h\in (0,1)}$ 
that converge in distribution to the general diffusion $Y$ on $\R$ with speed measure
$m(dx)=m_C(dx)+2\,dx$, where $m_C$ is the Cantor distribution. Such a process $Y$ can be understood as a Brownian motion
slowed down on the Cantor set.

For later reference we briefly recall a way to construct the Cantor distribution. 
To this end let $\mathcal C$ be the collection of all subsets of $[0,1]$ and 
let $\Psi\colon \mathcal C \to \mathcal C$ be the map given by
\begin{equation}
\Psi(A)=\frac{A}{3} \, \cup \frac{A+2}{3}, \quad A\subseteq [0,1].
\end{equation}
% as presented in \cite{Schmidt91}. To this end let $\mathcal C$ be the collection of all subsets of $[0,1]$ that are the union of finitely many disjoint closed intervals of $[0,1]$ and let $\Psi\colon \mathcal C \to \mathcal C$ be the map given by
%\begin{equation}
%\Psi\left(\bigcup_{i=1}^m [a_i,b_i]\right)=\bigcup_{i=1}^m \left(\left[ a_i, \frac{2a_i+b_i}{3}\right]\cup \left[\frac{a_i+2b_i}{3} , b_i\right] \right)
%\end{equation}
Next, we define recursively a sequence $(C_n)_{n\in \N_0}$
of subsets of $[0,1]$. Let $C_0=[0,1]$ and for $n\in \N$ let
\begin{equation}
C_n=\Psi(C_{n-1}).
\end{equation}
The Cantor set is defined as $C=\cap_{n\in \N}C_n$.

We define for all $n\in \N$ the probability measure $m_n$ on
$(\R, \mathcal B(\R))$ by
$m_n(dx)=\left(\frac{3}{2}\right)^n 1_{C_n}(x)\,dx$. Note that $m_n$ is absolutely continuous with respect to the Lebesgue measure $\mu_L$. It follows from the proof of Theorem~3.1 in \cite{Schmidt91}
that the sequence $(m_n)_{n\in \N}$ converges in distribution to a probability measure $m_C$
on $(\bbR,\cB(\bbR))$
and that for all $n\in \N$ it holds
\begin{equation}\label{eq:uniform_conv_cantor}
\sup_{x\in\bbR}|m_C((-\infty,x])-m_n((-\infty,x])|\le 2^{-(n-1)}.
\end{equation}
Moreover, it holds that $m_C(C)=1$
(in particular, $m_C$ is concentrated on $[0,1]$),
$\mu_L(C)=0$ and, for all $x\in \R$,
$m_C(\{x\})=0$, i.e., $m_C$ is a singular-continuous measure.

\begin{propo}\label{prop:21122018a1}
Let $m$ be the measure on $\R$ given by $m(dx)=m_C(dx)+2\,dx$ and let $Y$ be the associated diffusion. Let $n\colon (0,1)\to \N$ be a function satisfying 
$\lim_{h\to 0}2^{n(h)}{\sqrt{h}}=\infty$. 
Then there exists for all $h\in (0,1)$, $y\in \R$ a unique solution $a_h(y)\in (0,\sqrt{h}]$ of the equation
\begin{equation}\label{eq:eq_sf_cantor}
\frac{1}{2}\left(\frac{3}{2}\right)^{n(h)}\int_{y-a_h(y)}^{y+a_h(y)} 1_{C_{n(h)}}(u)(a_h(y)-|u-y|)du
+a^2_h(y)
=h.
\end{equation}
Let $(X^h)_{h\in (0,1)}$ be the family of Markov chains defined in
\eqref{eq:def_X} and
 \eqref{eq:13112017a1} (with scale factors $a_h$, $h\in(0,1)$, given by the solution of \eqref{eq:eq_sf_cantor}). Then for all $y\in \R$ the distributions of $(X^{h,y}_t)_{t\in [0,\infty)}$, $h\in (0,1)$, under $P$ converge weakly to the distribution of $(Y_t)_{t\in [0,\infty)}$ under $P_y$, as $h\to 0$.
\end{propo}

\begin{proof}
First observe that for all $y\in \R$ the mapping
\begin{equation}
[0,\infty) \ni a\mapsto \frac{1}{2}\left(\frac{3}{2}\right)^{n(h)}\int_{y-a}^{y+a} 1_{C_{n(h)}}(u)(a-|u-y|)du
+a^2 \in [0,\infty)
\end{equation}
is continuous and strictly increasing. This ensures existence of a unique solution $a_h(y)\in[0,\infty)$ of \eqref{eq:eq_sf_cantor}. It follows from
\begin{equation}
a^2_h(y)\le \frac{1}{2}\left(\frac{3}{2}\right)^{n(h)}\int_{y-a_h(y)}^{y+a_h(y)} 1_{C_{n(h)}}(u)(a_h(y)-|u-y|)du
+a^2_h(y)
=h
\end{equation}
that $a_h(y)\le \sqrt{h}$ for all $h\in (0,1)$, $y\in \R$. Moreover, it follows that for all $h\in (0,1)$ and $y\in \R$ it holds
\begin{equation}\label{eq:sf_cantor_aux1}
\begin{split}
&\frac{1}{2}\int_{(y-a_h(y),y+a_h(y))} (a_h(y)-|u-y|)\,(m_{n(h)}+2)(du)\\
&=
\frac{1}{2}\left(\frac{3}{2}\right)^{n(h)}\int_{y-a_h(y)}^{y+a_h(y)} 1_{C_{n(h)}}(u)(a_h(y)-|u-y|)\,du
+a^2_h(y)=h.
\end{split}
\end{equation}
Next, observe that formula~\eqref{eq:03012020a2}, definition~\eqref{eq:def_q} of $q$ and the fact that $m_C$ and $m_{n(h)}$ do not possess atoms ensure that it holds for all $h\in (0,1)$ and $y\in \R$ that
\begin{equation}
\begin{split}
&
\int_{(y-a_h(y),y+a_h(y))} (a_h(y)-|u-y|)\,(m_C-m_{n(h)})(du)\\
&=
\int_y^{y+a_h(y)}\left[m_C((y,u))-m_{n(h)}((y,u))\right]\,du+
\int_{y-a_h(y)}^y\left[m_C((u,y))-m_{n(h)}((u,y))\right]\,du
\\
&=
\int_y^{y+a_h(y)}\left[m_C((-\infty,u])-m_{n(h)}((-\infty,u])\right]\,du-
\int_{y-a_h(y)}^y\left[m_C((-\infty,u])-m_{n(h)}((-\infty,u])\right]\,du.
\end{split}
\end{equation}
This, \eqref{eq:uniform_conv_cantor} and the fact that $a_h(y)\le \sqrt{h}$ show that,
for all $h\in (0,1)$ and $y\in \R$, we have
\begin{equation}\label{eq:sf_cantor_aux2}
\begin{split}
&
\left|
\frac{1}{2}
\int_{(y-a_h(y),y+a_h(y))} (a_h(y)-|u-y|)\,(m_C-m_{n(h)})(du)
\right|
\\
&\le a_h(y)
\sup_{u\in\bbR} \left| m_C((-\infty,u])-m_{n(h)}((-\infty,u]) \right|
\le \sqrt{h} 2^{-(n(h)-1)}.
\end{split}
\end{equation}
Combining \eqref{eq:sf_cantor_aux1} and \eqref{eq:sf_cantor_aux2} and using the assumption $\lim_{h\to 0}2^{n(h)}{\sqrt{h}}=\infty$ shows that 
\begin{equation}
\begin{split}
&\sup_{y\in \R}\left| 
\frac{1}{2}\int_{(y-a_h(y),y+a_h(y))} (a_h(y)-|u-y|)\,m(du)-h
\right|\\
&\le 
\sup_{y\in \R} \left| 
\frac{1}{2}\int_{(y-a_h(y),y+a_h(y))} (a_h(y)-|u-y|)\,(m_C-m_{n(h)})(du)\right|\\
&\qquad+
\sup_{y\in \R}\left|\frac{1}{2}\int_{(y-a_h(y),y+a_h(y))} (a_h(y)-|u-y|)\,(m_{n(h)}+2)(du)-h
\right|\\
&=\sup_{y\in \R} \left| 
\frac{1}{2}\int_{(y-a_h(y),y+a_h(y))} (a_h(y)-|u-y|)\,(m_C-m_{n(h)})(du)\right|\in o(h).
\end{split}
\end{equation}
Hence, Condition~(A) is satisfied and weak convergence of $X^h$ to $Y$ follows from Theorem~\ref{main thm sec 1}.
\end{proof}

Proposition~\ref{prop:21122018a1}
provides the way to simulate approximations
of the Brownian trajectories
slowed down on the Cantor set
(see Figure~\ref{fig:cbm}).

\begin{figure}[htb!]
\centering
\includegraphics[width=0.49\textwidth]{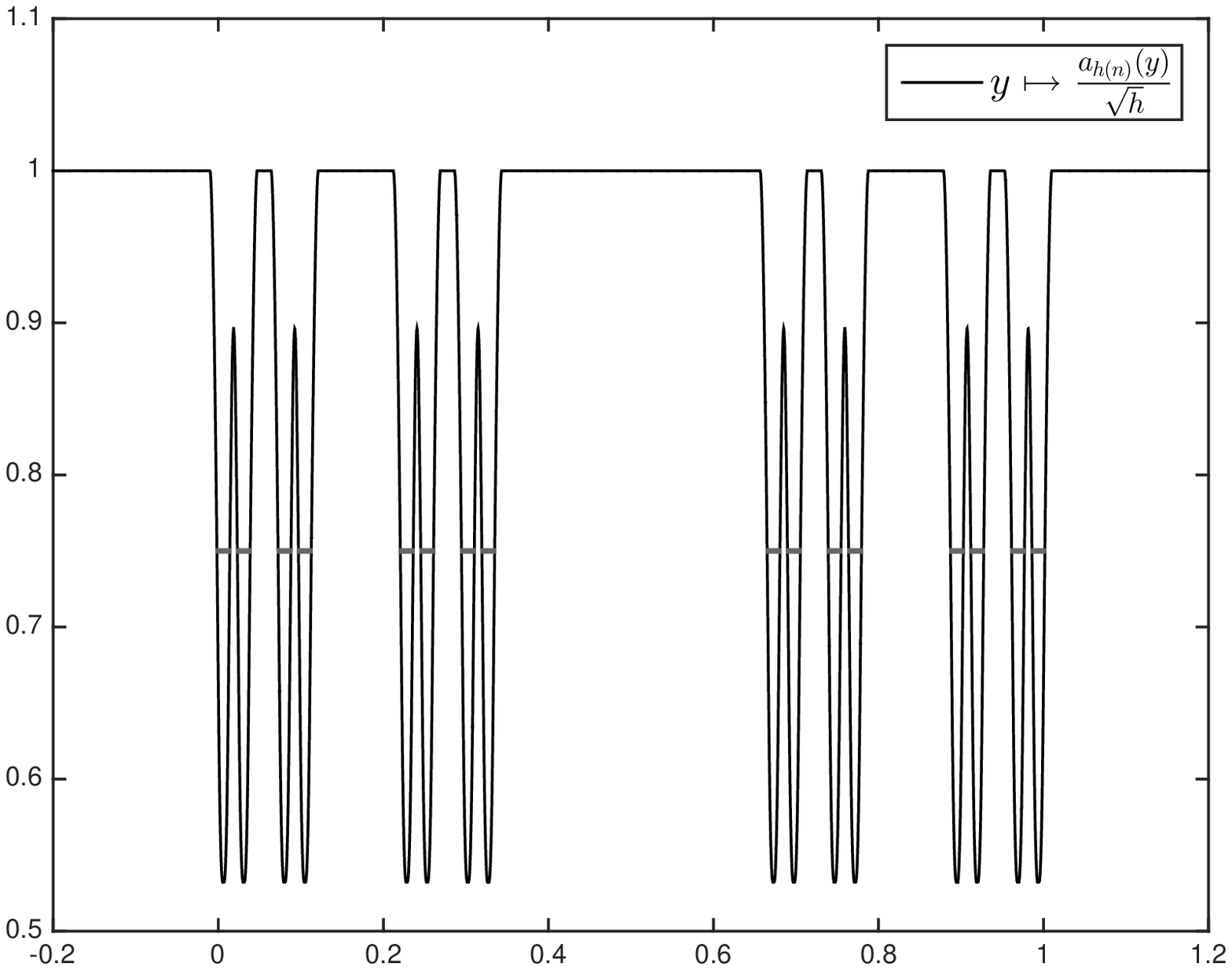}
\includegraphics[width=0.49\textwidth]{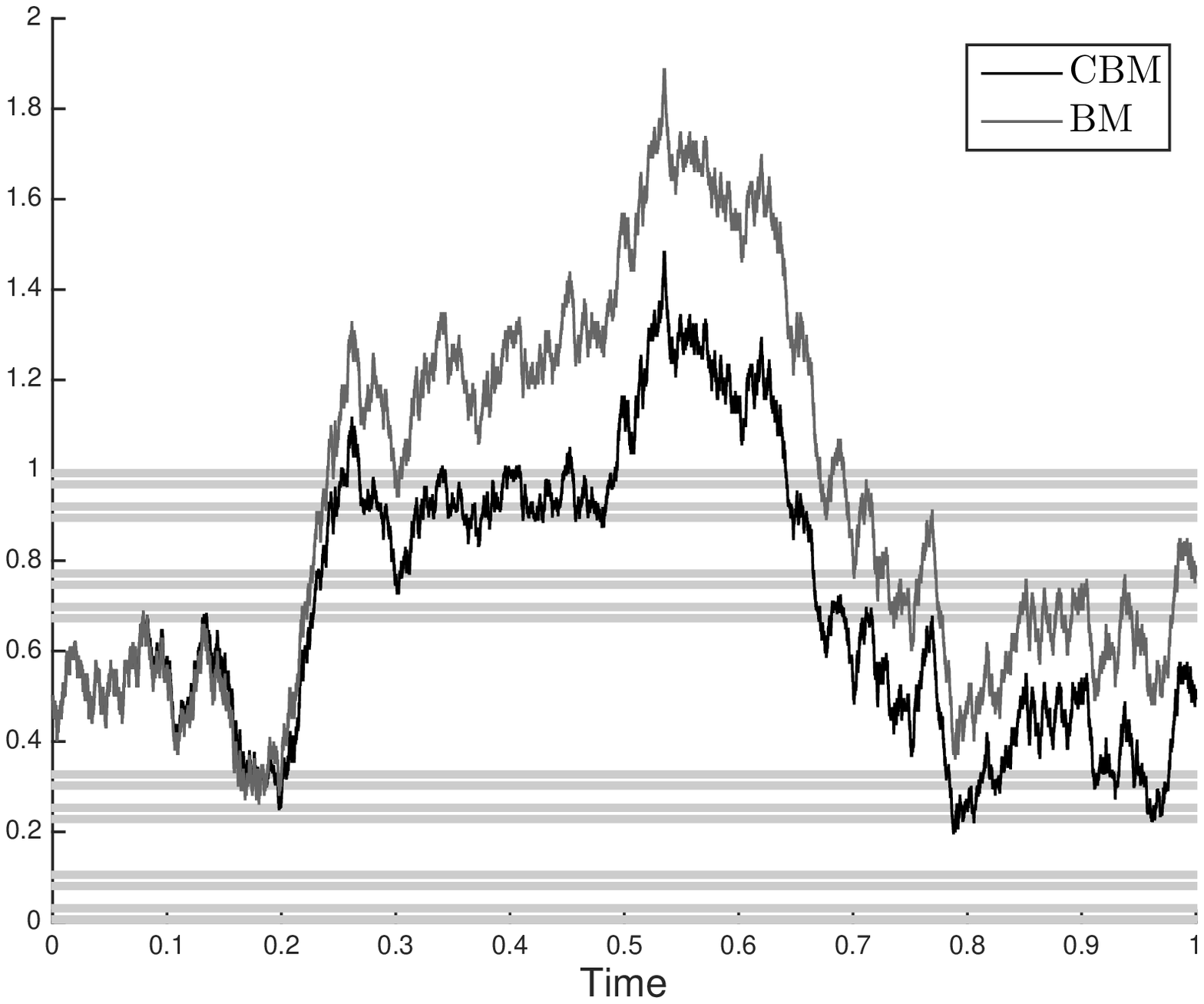}
{\caption{\footnotesize Left: Normalized scale factor for the approximation of the Brownian motion slowed down on the Cantor set. The black line depicts the function $\R \ni y\mapsto \frac{a_{n(h)}(y)}{\sqrt{h}}\in \R$, where for all $y\in \R$ the real number $a_{n(h)}(y)$ is the solution of \eqref{eq:eq_sf_cantor} with $h=10^{-4}$ and $n(h)=4$. The gray bars depict the 16 subintervals of the set  $C_{n(h)}$. We see that the black line
interpolates smoothly between 1 and considerably smaller values around the gray bars. 
Right: A realization of an approximation of the Brownian motion slowed down on the Cantor set (CBM) and of a Brownian motion (BM). The black line depicts a trajectory of $(X_t^h)_{t\in [0,1]}$ defined in
\eqref{eq:def_X} and
 \eqref{eq:13112017a1} with scale factor $a$ given by the solution of \eqref{eq:eq_sf_cantor} with $h=10^{-4}$ and $n(h)=4$. The gray line shows a realization of $(X_t^h)_{t\in [0,1]}$ defined in
\eqref{eq:def_X} and~\eqref{eq:13112017a1} with scale factor $a_h$ given by $a_h\equiv \sqrt{h}$ and $h=10^{-4}$. Both trajectories are generated from the same sample of random increments $(\xi_k)_{k\in \N}$. The horizontal light gray bars show again the 16 subintervals of the set  $C_{n(h)}$.
We see that on the set $C_{n(h)}$
%the CBM moves indeed with smaller volatility than BM.
the CBM is indeed slowed down (compared with the BM).}\label{fig:cbm}}
\end{figure}

\paragraph{Acknowledgement}
We thank the anonymous referee for comments that helped improve the exposition.
Thomas Kruse and Mikhail Urusov acknowledge
the support from the
\emph{German Research Foundation}
through the project 415705084.

\bibliographystyle{abbrv}
\bibliography{literature}
\end{document}